%% file: main.tex
\documentclass{amsart}
\usepackage{tikz}
\usepackage{thmtools,thm-restate}
\usepackage{ifthen}
\usepackage{scalefnt}
\usepackage[latin1]{inputenc}
\usepackage[T1]{fontenc}
\usepackage{amsmath,amssymb,amsthm,bm}
\usepackage{graphicx}
\usepackage{hyperref}
\usepackage{upgreek}
\usepackage{cleveref}

\usepackage[marginpar]{todo}

\usepackage{booktabs,siunitx}
\usepackage{scalerel}

\DeclareMathOperator{\snap}{snap}
\declaretheorem[numberwithin=section]{theorem}

\newcommand{\del}{\partial}
\newcommand{\abs}[1]{\left\lvert #1 \right\rvert}
\newcommand{\co}{\colon\thinspace}

\newcommand{\config}[3]{C(#1;#2,#3)}
\newcommand{\cell}[3]{X(#1;#2,#3)}
\newcommand{\ambient}[3]{G(#1;#2,#3)}

\newtheorem{lemma}[theorem]{Lemma}
\newtheorem{corollary}[theorem]{Corollary}
\newtheorem{conjecture}[theorem]{Conjecture}
\newtheorem{question}[theorem]{Question}

\theoremstyle{definition}

\newcommand{\R}{\mathbb{R}}

\newcommand{\Z}{\mathbb{Z}}

\begin{document}
\author[H. Alpert]{Hannah Alpert}
\thanks{H.A.\ was supported by NSF-DMS \#1802914 during work on this project.}
\author[U. Bauer]{Ulrich Bauer}
\thanks{U.B.\ gratefully acknowledges support from the German Research Foundation (DFG) through the Collaborative Research Center SFB/TRR 109 \emph{Discretization in Geometry and Dynamics.}}
\author[M. Kahle]{Matthew Kahle}
\thanks{M.K.\ gratefully acknowledges the support of NSF-DMS \#2005630, NSF-DMS \#1839358, and NSF-CCF \#1839358}
\author[R. MacPherson]{Robert MacPherson}
\author[K. Spendlove]{Kelly Spendlove}
\thanks{K.S.\ was partially supported by the NSF Graduate Research Fellowship Program under grant DGE-1842213 and by EPSRC grant EP/R018472/1.}
\title{Homology of configuration spaces of hard squares in a rectangle}
\maketitle

\begin{abstract}

We study ordered configuration spaces $\config{n}{p}{q}$ of $n$ hard squares in a $p \times q$ rectangle, a generalization of the well-known ``15 Puzzle''. Our main interest is in the topology of these spaces. Our first result is to describe a cubical cell complex and prove that is homotopy equivalent to the configuration space. We then focus on determining for which $n$, $j$, $p$, and $q$ the homology group $H_j [ \config{n}{p}{q} ]$ is nontrivial. We prove three homology-vanishing theorems, based on discrete Morse theory on the cell complex. Then we describe several explicit families of nontrivial cycles, and a method for interpolating between parameters to fill in most of the picture for ``large-scale'' nontrivial homology.

\end{abstract}

\section{Introduction}


We study the ordered configuration space of $n$ squares in a $p \times q$ rectangle, which we denote by $\config{n}{p}{q}$. The case $n=15$ and $p=q=4$ corresponds to the famous ``15 Puzzle''. This puzzle was apparently invented by Noyes Palmer Chapman, a postmaster in Canastota, New York, in 1874 \cite{15puzzle}. Already by 1879, the puzzle had been analyzed mathematically by Johnson and Story \cite{JS1879}. They showed that it is not possible, for example, for any sequence of moves to transpose the pieces labelled $14$ and $15$. Their observation is really a topological one, namely that the configuration space has two connected components.

A natural discrete model for the 15 Puzzle is the graph $G_{15}$, which we describe as follows. The vertices are the aligned positions of the puzzle, corresponding to the $16!$ permutations of the 15 pieces and the one hole, and we have an edge between every pair of positions that differ by sliding a piece into the hole. 

If we allow arbitrary positions for non-overlapping squares, as long as they do not overlap, then the configuration space for the 15 Puzzle is more than $1$-dimensional, though; for instance, there is a three-parameter family of ways to slide horizontally the three pieces in the bottom row. Nevertheless, as a special case of our results here, the configuration space of the 15 Puzzle deformation retracts to a one-dimensional subspace homeomorphic to $G_{15}$.

Having a cell complex structure allows for computing many topological invariants directly. For example, the Betti number $\beta_1$ can be computed by counting the number of $0$--cells $f_0=16!$ and $1$--cells $f_1=24 \times 15!$ of $G_{15}$, using the fact that $\beta_0 = 2$, and applying the $1$-dimensional Euler formula $f_0 - f_1 = \beta_0 - \beta_1$. 

In the more general setting, we describe a cubical complex $\cell{n}{p}{q}$ and show it is always a deformation retract of the configuration space $\config{n}{p}{q}$. Applying discrete Morse theory on this complex allows us to establish some necessary conditions on where nontrivial homology can appear.


In the following, we always assume that $p,q \ge 1$, $0 \le n \le pq$, and $j \ge 0$. We also sometimes use a ``large-scale'' parametrization, by defining $x=n/pq$ and $y=j/pq$.
The quantity $x$ has a physical interpretation as ``density'', describing the area ratio in the rectangular region that is occupied by squares.

\begin{restatable}[Homology vanishing theorem]{theorem}{vanishinghomology} \label{thm:homologyvanishing}
We have the following. 
\begin{enumerate}
    \item If $j > pq-n$, then $H_j[\config{n}{p}{q}] =0$.
    \item If $j > n$, then  $H_j[\config{n}{p}{q}] =0$.
    \item If $j > pq / 3$, then $H_j[\config{n}{p}{q}] =0$. 
\end{enumerate}
Equivalently, on the large scale we have that if $H_j [ \config{n}{p}{q} ] \neq 0$, then
 $y \le \min \{ 1-x, \; x, \; 1/ 3 \}.$ \\
\end{restatable}

The cubical cell complex model allows us to do exact computations for small examples. We include a table of Betti numbers in Section~\ref{sec:computation}. Based in part on our computations, we conjecture the following.

\begin{conjecture} \label{conj:main}
If $H_j [ \config{n}{p}{q} ] \neq 0$, then
\[j \le \min \{ pq-n, \; n-\frac{8n^2}{9pq}, \; pq/ 4 \}.\]
Equivalently, we conjecture that if
 $H_j [ \config{n}{p}{q} ] \neq 0$, then
 \[y \le \min \{ 1-x, \; x-(8/9)x^2, \; 1/ 4 \}.\]
\end{conjecture}

In Section \ref{sec:nontrivial}, we describe several families of explicit nontrivial cycles, and a method for interpolating between parameters. We \emph{almost} show that whenever  $y \le \min \{ 1-x, \; x-(8/9)x^2, \; 1/ 4 \}$, there exist $n,j,p,q$ such that $H_j [ \config{n}{p}{q}] \neq 0$.
Instead we prove an analogous statement with a piecewise linear approximation of the parabola $y = x -(8/9)x^2$. Let $S$ be the set of points on the parabola defined by
\[ S = \left\{ \left( x,x-(8/9)x^2 \right) \bigm\vert x=  \frac{3}{4k},k \ge 1 \right\}. \]
Note that $(3/4,1/4) \in S$ and $(3/8,1/4) \in S$.
Let $I$ be the closed interval $$ I = \{ (x,y) \mid 0 \le x \le 1 \mathrm{\ and\ } y = 0 \}.$$

\begin{restatable}[Large-scale homology non-vanishing theorem]{theorem}{nonvanish} \label{thm:allnonvanish}
Let $(x,y)$ be any rational point in the convex hull of $S \cup I$. Then there exist $n,p,q,j$ such that $x=n/pq$, $y=j/pq$, and $H_j[ \config{n}{p}{q} ] \neq 0$.
\end{restatable}

It may actually be that Theorem \ref{thm:allnonvanish} suggests the right necessary conditions for nontrivial homology, rather than Conjecture \ref{conj:main}. We do not currently know of any instance of $n,j,p,q$ where $H_j [\config{n}{p}{q}] \neq 0$ and $(x,y)$ lies outside of the convex hull of $S \cup I$.

A summary of our main results is illustrated in Figure \ref{fig:main}. Although we have made some headway, completely resolving the following question is left as future work. 

\begin{question} \label{quest:liq-sol} What are necessary and sufficient conditions on $(n; j ; p,q)$ for
\[ H_j [ \config{n}{p}{q}]  \neq 0? \]
\end{question}

\input{attained.txt}
We note that Conjecture \ref{conj:main} is only about necessary conditions for nontrivial homology, but at the moment we do not have a good conjecture for necessary and sufficient conditions. The conditions in Conjecture \ref{conj:main} by themselves are not sufficient. For example, $(1/4, 3/16)$ is a point in the blue region, corresponding to $n=p=q=4$, $j=3$.  However, it is not true that we have homology whenever $n/pq = 1/4$ , $j/pq = 3/16$, even when $n$ is sufficiently large.  Suppose that $p = 2$, $q = 8k$ for some $k \ge 1$, and $n = 4k$, then we cannot get nontrivial homology with $j=3k$.  The largest $j$ where we will see nontrivial homology is $j=2k$; by the homotopy equivalence mentioned below, this follows from  Theorem 1.2 (3) in \cite{AKM19}.

In recent years, there has been increased interest in similar kinds of configuration spaces; see \cite{Alpert17, BBK14, CGKM12} for some earlier work on configuration spaces of disks. Plachta recently studied configuration spaces of squares in a rectangle \cite{Plachta21}, using affine Morse--Bott theory, smooth flows, and graphs associated with such configurations. As one application, he shows that under certain conditions the configuration space which we denote $\config{n}{p}{q}$ is connected. We note that our dimensions of the rectangle $p$ and $q$ are always positive integers, but he studies the more general framework where they may be positive real numbers.  

What we study here is closely related to the recent paper \cite{AKM19} on configuration spaces of hard disks in an infinite strip, which we now briefly discuss. Let $\mathrm{config}(n,w)$ denote the configuration space of $n$ disks of unit diameter in an infinite strip of width $w$. While we do not prove it here, it is not hard to check that $\config{n}{p}{q}$ is homotopy equivalent to $\mathrm{config}(n,w)$ if $q \ge n$ and $p=w$. So the configuration spaces of hard squares in a rectangle we study here are a generalization of the configuration spaces of hard disks in an infinite strip.

Motivated by the notion of phase transitions for hard-spheres systems, definitions are suggested in \cite{AKM19} for homological solid, liquid, and gas regimes. The definitions apply here as well.

Let $\mathrm{Conf}(n;\R^2)$ denote the (ordered) configuration space of points in the plane. We say that $(n; j; p,q)$ is
\begin{itemize}
    \item in the \emph{homological solid regime} if
    \[H_j [ \config{n}{p}{q} ]= 0,\]
    
    \item in the \emph{homological gas regime} if the inclusion map $i: \config{n}{p}{q} \to \mathrm{Conf}(n; \R^2)$
    induces an isomorphism on homology
    $$i_* : H_j [ \config{n}{p}{q}] \to H_j [ \mathrm{Conf}(n ; \R^2)], \mbox{ and }$$
    
    \item in the \emph{homological liquid regime} otherwise.
\end{itemize}

In the present paper, we are mainly concerned with the boundary between trivial and nontrivial homology, i.e., separating the solid regime from liquid and gas. It will also be interesting to better understand the boundary between the homological liquid and gas regimes, as summarized in the following question.

\begin{question} \label{quest:liq-gas} What are necessary and sufficient conditions on $(n; j ; p,q)$ for the inclusion map $i: \config{n}{p}{q} \to \mathrm{Conf}(n; \R^2)$ to induce an isomorphism on homology
    \[i_* : H_j [ \config{n}{p}{q}] \to H_j [ \mathrm{Conf}(n ; \R^2)] ?\]
\end{question}

\section{Definitions and notation} \label{sec:defs}

The configuration space $\config{n}{p}{q}$ of $n$ unit squares in a $p$ by $q$ rectangle can be written as a subspace of $\mathbb{R}^{2n}$ by keeping track of the coordinates of the centers of the squares.  We select our $p$ by $q$ rectangle to be the set $\left[\frac{1}{2}, p + \frac{1}{2}\right] \times \left[\frac{1}{2}, q+\frac{1}{2}\right]$ in $\mathbb{R}^2$.  Accordingly, we define $\config{n}{p}{q}$ to be the set of all points $(x_1, y_1, \ldots, x_n, y_n) \in \mathbb{R}^{2n}$ such that
\begin{itemize}
\item $1 \leq x_k \leq p$ and $1 \leq y_k \leq q$ for all $1 \leq k \leq n$, and
\item $\abs{x_k - x_\ell} \geq 1$ or $\abs{y_k - y_\ell} \geq 1$ for all $1 \leq k < \ell \leq n$.
\end{itemize}
Note that the boundaries of the unit squares can intersect each other or the edges of the rectangle. 

We will be working with two ways to draw a grid on the rectangle; these two grids can be thought of as dual to each other, or as offset by $\left(\frac{1}{2}, \frac{1}{2}\right)$.  One grid is the integer coordinate grid.  The set of possible positions of the center of one square is $[1, p] \times [1, q]$, which we can think of as having vertices at the points where both coordinates are integers, edges between vertices at distance $1$, and $(p-1)(q-1)$ square $2$-cells.  We refer to these integer points as \emph{\textbf{coordinate grid vertices}}, to the edges as \emph{\textbf{coordinate grid edges}}, and to the squares as \emph{\textbf{coordinate grid squares}}.  Together we refer to the coordinate grid vertices, edges, and squares as \emph{\textbf{coordinate grid cells}}.

The other grid is the $p$ by $q$ grid on the rectangle itself.  Thinking of the rectangle as a $p$ by $q$ chessboard, we refer to the unit square centered at each coordinate grid vertex as a \emph{\textbf{board square}}.  For each coordinate grid cell, there is a corresponding rectangle of board squares given by taking the union of all unit squares for which the center lies on that coordinate grid cell, as shown in Figure~\ref{fig-grid-edge}.  The rectangle corresponding to a coordinate grid vertex is a single board square; the rectangle corresponding to a coordinate grid edge is a pair of adjacent board squares; and the rectangle corresponding to a coordinate grid square is a $2$ by $2$ rectangle of board squares.

\begin{figure}
\begin{center}
\begin{tikzpicture}[vert/.style={circle, draw=black, fill=black, inner sep = 0pt, minimum size = 1mm}]
\draw[draw=gray!40, fill=gray!40] (1.5, 1.5)--(2.5, 1.5)--(2.5, 3.5)--(1.5, 3.5)--(1.5, 1.5);
\draw (.5, .5)--(4.5, .5) (.5, 1.5)--(4.5, 1.5) (.5, 2.5)--(4.5, 2.5) (.5, 3.5)--(4.5, 3.5) (.5, .5)--(.5, 3.5) (1.5, .5)--(1.5, 3.5) (2.5, .5)--(2.5, 3.5) (3.5, .5)--(3.5, 3.5) (4.5, .5)--(4.5, 3.5);
\draw[dotted] (.5, 1)--(4.5, 1) (.5, 2)--(4.5, 2) (.5, 3)--(4.5, 3) (1, .5)--(1, 3.5) (2, .5)--(2, 3.5) (3, .5)--(3, 3.5) (4, .5)--(4, 3.5);
\draw (.5, .5)--(4.5, .5)--(4.5, 3.5)--(.5, 3.5)--(.5, .5);
\draw[very thick] (2, 2)--(2, 3);
\end{tikzpicture}
\end{center}
\caption{The coordinate grid vertices, at points with integer coordinates, are the centers of the board squares.  Here a coordinate grid edge is shown with its corresponding rectangular piece.}\label{fig-grid-edge}
\end{figure}

Let $\ambient{n}{p}{q}$ be the space $([1, p] \times [1, q])^n$ with its standard cubical complex structure.  Here the letter $G$ stands for \emph{grid}.  We can think of this space as the set of configurations of labeled squares in the rectangle where the squares are allowed to overlap.  As a cubical complex, each cell of $\ambient{n}{p}{q}$ corresponds to an $n$-tuple in which each entry is a coordinate grid cell.  We can draw the cell of $\ambient{n}{p}{q}$ by drawing the $n$ corresponding rectangles of board squares.  We refer to such a picture as a \emph{\textbf{rectangle arrangement}}, and we refer to the $n$ rectangles as \emph{\textbf{pieces}} in the arrangement.  Any list of $n$ rectangles of board squares of sizes $1$ by $1$, $1$ by $2$, $2$ by $1$, and $2$ by $2$ is the rectangle arrangement of some cell of $\ambient{n}{p}{q}$.

\input{npq222.txt}

We define $\cell{n}{p}{q}$ to be the subcomplex of $\ambient{n}{p}{q}$ consisting of all cells of $\ambient{n}{p}{q}$ that are fully contained in $\config{n}{p}{q}$.  Here the letter $X$ stands for \emph{complex}, because $\cell{n}{p}{q}$ is the main cell complex that we work with throughout the paper.  It is quick to check that $\cell{n}{p}{q}$ is equal to the set of cells in which the corresponding rectangle arrangement has none of its pieces overlapping.  Given a configuration in $\config{n}{p}{q}$, we can check whether it is in $\cell{n}{p}{q}$ by looking at each unit square in the configuration and drawing the rectangle of board squares that it intersects, as shown in Figure~\ref{fig-rectangles}.  If these rectangles are disjoint, then the configuration is in $\cell{n}{p}{q}$, and it is in the cell corresponding to the rectangular arrangement that we have just drawn.

\begin{figure}
\begin{center}
\begin{tikzpicture}
\draw (.5, .5)--(4.5, .5) (.5, 1.5)--(4.5, 1.5) (.5, 2.5)--(4.5, 2.5) (.5, 3.5)--(4.5, 3.5) (.5, .5)--(.5, 3.5) (1.5, .5)--(1.5, 3.5) (2.5, .5)--(2.5, 3.5) (3.5, .5)--(3.5, 3.5) (4.5, .5)--(4.5, 3.5);
\draw[thick, fill=gray!40] (1.5, 1.7)--(2.5, 1.7)--(2.5, 2.7)--(1.5, 2.7)--cycle (3.1, 1.2)--(4.1, 1.2)--(4.1, 2.2)--(3.1, 2.2)--cycle;
\end{tikzpicture}
\begin{tikzpicture}
\draw (.5, .5)--(4.5, .5) (.5, 1.5)--(4.5, 1.5) (.5, 2.5)--(4.5, 2.5) (.5, 3.5)--(4.5, 3.5) (.5, .5)--(.5, 3.5) (1.5, .5)--(1.5, 3.5) (2.5, .5)--(2.5, 3.5) (3.5, .5)--(3.5, 3.5) (4.5, .5)--(4.5, 3.5);
\draw[draw=gray!40, fill=gray!40] (1.6, 1.6)--(2.4, 1.6)--(2.4, 3.4)--(1.6, 3.4)--cycle (2.6, .6)--(4.4, .6)--(4.4, 2.4)--(2.6, 2.4)--cycle;
\end{tikzpicture}
\end{center}
\caption{Any configuration where no two squares touch the same board square is in the cell of $\cell{n}{p}{q}$ corresponding to the rectangle arrangement that shows which board squares each square touches.}\label{fig-rectangles}
\end{figure}

\section{Homotopy equivalence of the configuration space and complex} \label{sec:homotopy}
The ambient cubical complex $\ambient{n}{p}{q}$ has three kinds of cells: some cells are fully contained in $\config{n}{p}{q}$, and together form $\cell{n}{p}{q}$; some cells are partially in $\config{n}{p}{q}$; and some cells are disjoint from $\config{n}{p}{q}$.  We will define a deformation retraction from $\config{n}{p}{q}$ to $\cell{n}{p}{q}$ by considering the cells of $\ambient{n}{p}{q}$ that are partially in $\config{n}{p}{q}$ one at a time.  To do this, we define local coordinates on each of these cells and give a criterion in those local coordinates for which points are in $\config{n}{p}{q}$ and which points are not.

We define a function $\snap \co \mathbb{R} \rightarrow \mathbb{R}$ by $\snap(x) = \frac{1}{2}(\lfloor x \rfloor + \lceil x \rceil)$.  In other words, we have $\snap(k) = k$ for all $k \in \mathbb{Z}$, and if $x \in (k, k+1)$, then $\snap(x) = k + \frac{1}{2}$.  We can also define $\snap \co \mathbb{R}^d \rightarrow \mathbb{R}^d$ for any dimension $d$, by applying snap to each coordinate separately.

If $z = (x_1, y_1, \ldots, x_n, y_n) \in \mathbb{R}^{2n}$ is a point in the complex $\ambient{n}{p}{q}$, then $\snap(z)$ is the barycenter of the unique cubical cell of $\ambient{n}{p}{q}$ whose interior contains $z$.  Geometrically, if $(x_i, y_i)$ is the center of a unit square, then $\snap(x_i, y_i)$ is the center of the corresponding rectangle of the board squares that it touches.
Note that $\snap$ is idempotent, $\snap(\snap(z)) = \snap(z)$, and $z$ is a barycenter of some grid cell in $\ambient{n}{p}{q}$ if and only if $z = \snap(z)$.

\subsection{Containment of cells in the configuration space}

We can check whether a given cell of $\ambient{n}{p}{q}$ has empty intersection with $\config{n}{p}{q}$ by looking at pairs of pieces, case by case, in its corresponding rectangle arrangement.  Figure~\ref{fig-noconfigs} shows which pairs of pieces prevent a cell from having any configurations in $\config{n}{p}{q}$; in each case, the barycenter of the corresponding cell is not a configuration in $\config{n}{p}{q}$.  For each pair of pieces, there is no way to fit a unit square in the interior of each piece, while keeping the two unit squares disjoint.  (Two unit squares can fit if they touch the boundaries of the rectangles, but the resulting configuration is in the boundary of the specified open cell, not inside it.)  

Figure~\ref{fig-allowed-overlap} shows the four remaining ways for two pieces in a rectangle arrangement to overlap; for these, the corresponding cell is partially in in $\config{n}{p}{q}$, and 
the barycenter is a configuration in $\config{n}{p}{q}$.  
The following lemma summarizes how to check whether a given cell of $\ambient{n}{p}{q}$ is partially in $\config{n}{p}{q}$.

\begin{figure}
\begin{center}
\begin{tikzpicture}[scale=.7]
\draw (-.5, 0)--(2.5, 0) (-.5, 1)--(2.5, 1) (0, -.5)--(0, 1.5) (1, -.5)--(1, 1.5) (2, -.5)--(2, 1.5);
\draw[draw=gray!40,fill=gray!40] (.1, .1)--(1.9, .1)--(1.9, .9)--(.1, .9)--cycle;
\draw[draw=gray!80,fill=gray!80] (.1, .1)--(.9, .1)--(.9, .9)--(.1, .9)--cycle;
\draw[fill=black] 
(.5, .5) circle (.03)
(1, .5) circle (.03);
\end{tikzpicture}
\begin{tikzpicture}[scale=.7]
\draw (-.5, 0)--(2.5, 0) (-.5, 1)--(2.5, 1) (-.5, 2)--(2.5, 2) (0, -.5)--(0, 2.5) (1, -.5)--(1, 2.5) (2, -.5)--(2, 2.5);
\draw[draw=gray!40,fill=gray!40] (.1, .1)--(1.9, .1)--(1.9, 1.9)--(.1, 1.9)--cycle;
\draw[draw=gray!80,fill=gray!80] (.1, .1)--(.9, .1)--(.9, .9)--(.1, .9)--cycle;
\draw[fill=black] 
(.5, .5) circle (.03)
(1, 1) circle (.03);
\end{tikzpicture}
\begin{tikzpicture}[scale=.7]
\draw (-.5, 0)--(2.5, 0) (-.5, 1)--(2.5, 1) (-.5, 2)--(2.5, 2) (0, -.5)--(0, 2.5) (1, -.5)--(1, 2.5) (2, -.5)--(2, 2.5);
\draw[draw=gray!40,fill=gray!40] (.1, .1)--(1.9, .1)--(1.9, 1.9)--(.1, 1.9)--cycle;
\draw[draw=gray!80,fill=gray!80] (.1, .1)--(1.9, .1)--(1.9, .9)--(.1, .9)--cycle;
\draw[fill=black]
(1, .5) circle (.03)
(1, 1) circle (.03);
\end{tikzpicture}
\begin{tikzpicture}[scale=.7]
\draw (-.5, 0)--(2.5, 0) (-.5, 1)--(2.5, 1) (-.5, 2)--(2.5, 2) (0, -.5)--(0, 2.5) (1, -.5)--(1, 2.5) (2, -.5)--(2, 2.5);
\draw[draw=gray!40,fill=gray!40] (.1, 1.1)--(1.9, 1.1)--(1.9, 1.9)--(.1, 1.9)--cycle (1.1, .1)--(1.9, .1)--(1.9, 1.9)--(1.1, 1.9)--cycle;
\draw[draw=gray!80,fill=gray!80] (1.1, 1.1)--(1.9, 1.1)--(1.9, 1.9)--(1.1, 1.9)--cycle;
\draw[fill=black]
(1.5, 1) circle (.03)
(1, 1.5) circle (.03);
\end{tikzpicture}
\end{center}
\caption{If the unit squares at the centers of the rectangular pieces overlap, then the corresponding cell in $\ambient{n}{p}{q}$ does not contain any configurations in $\config{n}{p}{q}$.  The darker gray indicates where the two pieces overlap, and the black dots give the centers of the pieces.}\label{fig-noconfigs}
\end{figure}

\begin{figure}
\begin{center}
\begin{tikzpicture}[scale=.7]
\draw (-.5, 0)--(3.5, 0) (-.5, 1)--(3.5, 1) (0, -.5)--(0, 1.5) (1, -.5)--(1, 1.5) (2, -.5)--(2, 1.5) (3, -.5)--(3, 1.5);
\draw[draw=gray!40,fill=gray!40] (.1, .1)--(1.9, .1)--(1.9, .9)--(.1, .9)--cycle (1.1, .1)--(2.9, .1)--(2.9, .9)--(1.1, .9)--cycle;
\draw[draw=gray!80,fill=gray!80] (1.1, .1)--(1.9, .1)--(1.9, .9)--(1.1, .9)--cycle;
\draw[fill=black]
(1, .5) circle (.03)
(2, .5) circle (.03);
\end{tikzpicture}
\begin{tikzpicture}[scale=.7]
\draw (-.5, 0)--(3.5, 0) (-.5, 1)--(3.5, 1) (-.5, 2)--(3.5, 2) (0, -.5)--(0, 2.5) (1, -.5)--(1, 2.5) (2, -.5)--(2, 2.5) (3, -.5)--(3, 2.5);
\draw[draw=gray!40,fill=gray!40] (.1, .1)--(1.9, .1)--(1.9, .9)--(.1, .9)--cycle (1.1, .1)--(2.9, .1)--(2.9, 1.9)--(1.1, 1.9)--cycle;
\draw[draw=gray!80,fill=gray!80] (1.1, .1)--(1.9, .1)--(1.9, .9)--(1.1, .9)--cycle;
\draw[fill=black]
(1, .5) circle (.03)
(2, 1) circle (.03);
\end{tikzpicture}
\begin{tikzpicture}[scale=.7]
\draw (-.5, 0)--(3.5, 0) (-.5, 1)--(3.5, 1) (-.5, 2)--(3.5, 2) (0, -.5)--(0, 2.5) (1, -.5)--(1, 2.5) (2, -.5)--(2, 2.5) (3, -.5)--(3, 2.5);
\draw[draw=gray!40,fill=gray!40] (.1, .1)--(1.9, .1)--(1.9, 1.9)--(.1, 1.9)--cycle (1.1, .1)--(2.9, .1)--(2.9, 1.9)--(1.1, 1.9)--cycle;
\draw[draw=gray!80,fill=gray!80] (1.1, .1)--(1.9, .1)--(1.9, 1.9)--(1.1, 1.9)--cycle;
\draw[fill=black]
(1, 1) circle (.03)
(2, 1) circle (.03);
\end{tikzpicture}
\begin{tikzpicture}[scale=.7]
\draw (-.5, 0)--(3.5, 0) (-.5, 1)--(3.5, 1) (-.5, 2)--(3.5, 2) (-.5, 3)--(3.5, 3) (0, -.5)--(0, 3.5) (1, -.5)--(1, 3.5) (2, -.5)--(2, 3.5) (3, -.5)--(3, 3.5);
\draw[draw=gray!40,fill=gray!40] (.1, .1)--(1.9, .1)--(1.9, 1.9)--(.1, 1.9)--cycle (1.1, 1.1)--(2.9, 1.1)--(2.9, 2.9)--(1.1, 2.9)--cycle;
\draw[draw=gray!80,fill=gray!80] (1.1, 1.1)--(1.9, 1.1)--(1.9, 1.9)--(1.1, 1.9)--cycle;
\draw[fill=black]
(1, 1) circle (.03)
(2, 2) circle (.03);
\end{tikzpicture}
\end{center}
\caption{If a given cell of $\ambient{n}{p}{q}$ is partially contained in $\config{n}{p}{q}$, then some pair of overlapping pieces in the rectangle arrangement must overlap in one of the four ways shown.  The darker gray indicates where the two pieces overlap, and the black dots give the centers of the pieces.}\label{fig-allowed-overlap}
\end{figure}

\begin{lemma}\label{lem-partial-cell}
Let $z = \snap(z) = (x_1, y_1, \ldots, x_n, y_n) \in \ambient{n}{p}{q}$ be the barycenter of an open cell $\sigma$ of $\ambient{n}{p}{q}$.  Then
\begin{enumerate}
\item $\sigma$ has a nonempty intersection with $\config{n}{p}{q}$
if and only if its barycenter~$z$ lies in the configuration space $\config{n}{p}{q}$, or equivalently, the $\ell^\infty$ distance between $(x_\ell,y_\ell)$ and $(x_k,y_k)$ is at least $1$ for all $1 \leq k < \ell \leq n$:
\[\max(\abs{x_\ell - x_k}, \abs{y_\ell - y_k}) \geq 1.\]
\item $\sigma$ is fully contained in $\config{n}{p}{q}$, and hence a cell of $\cell{n}{p}{q}$, if and only if for all $1 \leq k < \ell \leq n$,
the corresponding pieces do not overlap, or equivalently,
\[\lfloor \max(x_k, x_\ell) \rfloor > \lceil \min(x_k, x_\ell) \rceil 
\quad
\text{or}
\quad
\lfloor \max(y_k, y_\ell) \rfloor > \lceil \min(y_k, y_\ell) \rceil.\]
 
\end{enumerate}
\end{lemma}

\begin{proof}
%
%
%
To check the first claim, we observe that if any point $z \in \sigma$ is in $\config{n}{p}{q}$, then $\snap(z) \in \config{n}{p}{q}$ as well.  This is because for any $x_1, x_2 \in \mathbb{R}$, if $x_2 - x_1 \geq 1$, then $\snap(x_2) - \snap(x_1) \geq 1$ as well.  

For the second statement, note that piece $k$ covers the board squares with centers in $[\lfloor x_k \rfloor, \lceil x_k \rceil] \times [\lfloor y_k \rfloor, \lceil y_k \rceil]$, and piece $\ell$ covers the board squares with centers in $[\lfloor x_\ell \rfloor, \lceil x_\ell \rceil] \times [\lfloor y_\ell \rfloor, \lceil y_\ell \rceil]$.  The two pieces overlap if and only if the intervals $[\lfloor x_k \rfloor, \lceil x_k \rceil]$ and $[\lfloor x_\ell \rfloor, \lceil x_\ell \rceil]$ overlap and the intervals $[\lfloor y_k \rfloor, \lceil y_k \rceil]$ and $[\lfloor y_\ell \rfloor, \lceil y_\ell \rceil]$ also overlap.
\end{proof}

We say that a subcomplex of a regular CW complex is a \emph{full subcomplex} if it is maximal with respect to its vertex set.

\begin{corollary}\label{cor-full-subcomplex}
$\cell{n}{p}{q}$ is a full subcomplex of the ambient cubical complex $\ambient{n}{p}{q}$.
\end{corollary}
%
%
An equivalent description for when an open cell $\sigma$ is partially in $\config{n}{p}{q}$ can be obtained from examining the cases in Figure~\ref{fig-allowed-overlap}:
Let $b = (i_1, j_1, \ldots, i_n, j_n) \in \ambient{n}{p}{q}$ be the barycenter of $\sigma$; note that the coordinates $i_k, j_k$ are half-integers.
Then $\sigma$ is partially in $\config{n}{p}{q}$ if and only if
\begin{enumerate}
\item for all $k$ and $\ell$, we have
$\max(\abs{i_\ell - i_k}, \abs{j_\ell- j_k}) \geq 1,$ and
\item 
there is a pair $k, \ell$ such that either 
\begin{enumerate}
\item
$\abs{i_\ell - i_k} = 1$ and $\abs{j_\ell - j_k} < 1$ and $i_k, i_\ell$ are not integers, or 
\item 
$\abs{j_\ell - j_k} = 1$ and $\abs{i_\ell - i_k} < 1$ and $j_k, j_\ell$ are not integers, or 
\item
$\abs{i_\ell - i_k} = \abs{j_\ell - j_k} = 1$ and none of $i_k, i_\ell, j_k, j_\ell$ are integers.
\end{enumerate}
\end{enumerate}

\subsection{Membership in the configuration space using local coordinates}

The next lemma specifies how to use local coordinates to check, for an open cell partially in $\config{n}{p}{q}$, whether a given point in the cell is in $\config{n}{p}{q}$.  Given an open cell $\sigma$ of $\ambient{n}{p}{q}$, we can specify the points $z \in \sigma$ in terms of the local coordinates $z - \snap(z) \in (-\frac{1}{2}, \frac{1}{2} )^{2n}$.  Not every point in $(-\frac{1}{2}, \frac{1}{2})^{2n}$ corresponds to a point in the cell, because for each coordinate of the barycenter $\snap(z)$ that is an integer, the corresponding coordinate in $z - \snap(z)$ is zero.

Let $b$ be the barycenter of cell $\sigma$, and let $I(\sigma)$ be the set of indices of non-integer coordinates of $b$.  The number of elements of $I(\sigma)$ is the dimension of $\sigma$.  Let $(-\frac{1}{2}, \frac{1}{2})^{I(\sigma)}$ denote the coordinate subspace of $(-\frac{1}{2}, \frac{1}{2})^{2n}$ given by letting the $I(\sigma)$ coordinates vary and requiring the remaining coordinates (corresponding to the integer coordinates in $b$) to be zero.  We have $z \in \sigma$ if and only if $z-b \in (-\frac{1}{2}, \frac{1}{2})^{I(\sigma)}$, in which case $b = \snap(z)$.  

A point of $\ambient{n}{p}{q}$ is in $\config{n}{p}{q}$ if and only if no two of the $n$ specified squares intersect.  Thus, we check the local coordinates for two of the $n$ squares at a time to see whether those two squares overlap.

\begin{lemma}\label{lem-partial-point}
Let $\sigma$ be an open cell of $\ambient{2}{p}{q}$ that is partially in $\config{2}{p}{q}$, and let $z = (x_1, y_1, x_2, y_2)$ be a point in the interior of $\sigma$. 
Then $(x_1, y_1, x_2, y_2) \in \config{2}{p}{q}$ if and only if one of the following conditions holds: 
\begin{itemize}
\item $\abs{\snap(x_2) - \snap(x_1)} = 1$ and
$(x_2 - \snap(x_2)) - (x_1 - \snap(x_1))$ is zero or
has the same sign as $\snap(x_2) - \snap(x_1)$, or
\item $\abs{\snap(y_2) - \snap(y_1)} = 1$ and
$(y_2 - \snap(y_2)) - (y_1 - \snap(y_1))$ is zero or
has the same sign as $\snap(y_2) - \snap(y_1)$.
\end{itemize}
\end{lemma}

In the fourth case in Figure~\ref{fig-allowed-overlap}, where the two pieces are $2$ by $2$ rectangles intersecting at one board square, either condition in the lemma may hold, so the intersection of $\config{n}{p}{q}$ with the cell of $\ambient{2}{p}{q}$ is the union of solutions to two linear inequalities.  In the other three cases, the centers of the two pieces have only one coordinate that differs by $1$, so the intersection of $\config{n}{p}{q}$ with the cell is the set of solutions to one linear inequality.

\begin{proof}
A point $(x_1, y_1, x_2, y_2)$ is in $\config{2}{p}{q}$ if and only if either $\abs{x_2 - x_1} \geq 1$ or $\abs{y_2 - y_1} \geq 1$.
Note that the function $\snap$ is weakly order-preserving, meaning that 
$x_2 - x_1 \geq 0$ implies $\snap(x_2) - \snap(x_1) \geq 0$. 
Thus, by symmetry of $x_1$ and $x_2$ as well as $(x_1,x_2)$ and $(y_1,y_2)$, it suffices to show that $x_2 - x_1 \geq 1$ if and only if both $\snap(x_2) - \snap(x_1) = 1$ and $(x_2 - \snap(x_2)) - (x_1 - \snap(x_1)) \geq 0$.
The latter condition straightforwardly implies the former.

Conversely, assume  $x_2 - x_1 \geq 1$.
Then clearly $\snap(x_2) - \snap(x_1) \geq 1$.
Moreover, the assumption that $\sigma$ is only partially in $\config{2}{p}{q}$ rules out the case $\snap(x_2) - \snap(x_1) > 1$,
as in this case we would necessarily have $\lfloor \snap(x_2) \rfloor > \lceil \snap(x_1) \rceil$, and \Cref{lem-partial-cell} would imply that $\sigma$ is fully contained in $\config{2}{p}{q}$. 
Thus we get $\snap(x_2) - \snap(x_1) = 1$ and $(x_2 - \snap(x_2)) - (x_1 - \snap(x_1)) \geq 0$ as desired.
\end{proof}

\subsection{Construction of the deformation retraction}

The next lemma gives the main step in constructing the deformation retraction from $\config{n}{p}{q}$ to $\cell{n}{p}{q}$.

\begin{lemma}
\label{lem:defRetractCell}
Let $\sigma$ be an open cell of $\ambient{n}{p}{q}$ that is partially in $\config{n}{p}{q}$.  Then we have that~$\del\sigma \cap \config{n}{p}{q}$ is a deformation retract of~$\overline{\sigma} \cap \config{n}{p}{q}$.
\end{lemma}

\begin{proof}
Let $z = (x_1, y_1, \ldots, x_n, y_n)$ be a point in the open cell $\sigma$.
If we want to check whether $z$ is in $\config{n}{p}{q}$, then Lemma~\ref{lem-partial-point} gives a set of inequalities on the local coordinates $z - \snap(z) = (u_1, v_1, \ldots, u_n, v_n)$ within the open cell $\sigma$ that we can evaluate.  For each pair of pieces $k, \ell$ in the rectangle arrangement for $\sigma$, the lemma specifies zero, one, or two inequalities of the following form:
\begin{itemize}
\item $u_k \geq u_\ell$ or $u_k \leq u_\ell$;
\item $v_k \geq v_\ell$ or $v_k \leq v_\ell$.
\end{itemize}
The case of zero inequalities comes from the case where the two rectangular pieces do not intersect.  The case of one inequality comes from the case where they intersect in one of the first three ways shown in Figure~\ref{fig-allowed-overlap}.  And, the case of two inequalities comes from the fourth case in Figure~\ref{fig-allowed-overlap}, where the pieces are both $2$ by $2$ squares and they have one board square in common; in this case, the local coordinate $z - \snap(z)$ needs to satisfy either or both of the two inequalities.  Together, we refer to the inequalities from the lemma as the inequalities associated to $\sigma$.
Note that the above inequalities are stated for local coordinates, but at the same time,
the coordinates of the barycenter
\[
b=(i_1,j_1,\dots,i_n,j_n)
\]
satisfy the same set of inequalities, even strictly.
This property will be crucial for our argument.

We define the deformation retraction  
as follows.
Let $\lambda$ be large enough that $\frac{1}{\lambda} b$ is in $\left(-\frac12,\frac12\right)^{2n}$; we can take $\lambda = 2(p+q)$.  Let $m$ be the coordinate projection of $\frac{1}{\lambda}b$ onto $\left(-\frac{1}{2}, \frac{1}{2}\right)^{I(\sigma)}$; that is, for each integer coordinate of $b$, we set the corresponding coordinate of $m$ to be zero.  
Since $m$ is a positively scaled version of $b$, it inherits the magical quality of satisfying all of the inequalities associated to $\sigma$, and since $b$ satisfies all those inequalities strictly, $-m$ has the magical quality of violating all of the inequalities associated to $\sigma$.  (Note that every coordinate appearing in the inequalities associated to $\sigma$ is in $I(\sigma)$.)  Thus the point $b+m$ in $\sigma$ is in the configuration space $\config{n}{p}{q}$, while the point $b-m$ is not.

The deformation retraction now pushes every point $z \in \sigma$ outward along a ray from $b-m$ until it hits $\del \sigma$.  In other words, the vector from $b-m$ to $z$ is given by $z - b + m$, so as time~$t$ increases from~$0$, we set
\[z_t = z + t(z-b+m),\]
until we reach the maximum~$t$ for which~$z + t(z-b+m)$ is in~$\overline{\sigma}$, and then the point no longer moves.  Formally, we can define $T_z$ to be the positive value such that $z + T_z(z-b+m) \in \del\sigma$.  (Because $\sigma$ is a cube and hence star-shaped around any interior point, for any point in $\sigma$ and any nonzero vector within $\sigma$ starting at that point,  there is a unique non-negative multiple of that vector that reaches $\del\sigma$.)  If $d$ is the distance within $\sigma$ from $b-m$ to $\config{n}{p}{q}$, then the vector $z-b+m$ from $b-m$ to $z$ has length at least $d$, and any vector from $z$ to $\del \sigma$ has length at most $\mathrm{diam}(\sigma)$, so $T_z \leq \frac{1}{d} \mathrm{diam}(\sigma) \leq \frac{2n}{d}$.

Using this notation, the deformation retraction is defined as
\begin{align*}
F \co \overline{\sigma} \cap \config{n}{p}{q} \times \left[0,\frac{2n}{d}\right] &\to \overline{\sigma} \cap \config{n}{p}{q}, \\ 
(z,t) &\mapsto z + \min(t,T_z)(z-b+m).
\end{align*}

We still need to check that if $z \in \config{n}{p}{q}$, then $z_t \in \config{n}{p}{q}$ for all $t$, in order to ensure that the homotopy remains in $\overline{\sigma} \cap \config{n}{p}{q}$.  This is equivalent to checking that $$z_t-b= (z-b) + t(z-b + m) = (1 + t)(z-b) + tm$$
satisfies a sufficient collection of the inequalities associated to $\sigma$.  We claim that $z_t - b$ satisfies every one of the inequalities that $z-b$ satisfies; since this collection of inequalities is sufficient for $z$ to be in $\config{n}{p}{q}$, it is also sufficient for $z_t$ to be in $\config{n}{p}{q}$.  Indeed, the inequalities are linear with no constant term, so given two points satisfying the inequalities, any linear combination of them with positive coefficients also satisfies the inequalities.  Because $z-b$ satisfies a sufficient set of inequalities and $m$ satisfies all of the inequalities associated to $\sigma$, this implies that $(1 + t)(z - b) + tm $ also satisfies the same set of inequalities as $z-b$.  Thus, $z_t$ is in $\config{n}{p}{q}$ for every $t \leq T_z$, and the map $F$ that we have defined is indeed a deformation retraction from~$\overline{\sigma} \cap \config{n}{p}{q}$ to~$\del\sigma \cap \config{n}{p}{q}$.
\end{proof}

Putting all the cells together, we obtain a deformation retraction from $\config{n}{p}{q}$ to $\cell{n}{p}{q}$.

\begin{theorem}
\label{thm:defRetract}
The subcomplex~$\cell{n}{p}{q}$ is a deformation retract of the configuration space~$\config{n}{p}{q}$.
\end{theorem}

\begin{proof}
Order the cells~$\sigma$ of~$\ambient{n}{p}{q}$ that are partially in $\config{n}{p}{q}$, such that their dimensions are nonincreasing.  Then, cell by cell in order, we use \Cref{lem:defRetractCell} to obtain a deformation retraction from~$\overline{\sigma} \cap \config{n}{p}{q}$ to~$\del\sigma \cap \config{n}{p}{q}$.  Concatenating these deformation retractions gives a deformation retraction from~$\config{n}{p}{q}$ to the set~$\cell{n}{p}{q}$ of cells completely contained in $\config{n}{p}{q}$.
\end{proof}

\section{Discrete Morse theory} \label{sec:dMt}

In this section, we describe a discrete gradient vector field on $\cell{n}{p}{q}$, in the sense of Forman's discrete Morse theory \cite{Forman1998Morse}, and characterize its critical cells.  The analysis of which cells are critical is based on what we call the \textbf{\textit{apex}} of a cell.  The apex of a cell, as shown in Figure~\ref{apexfig}, is the $0$-dimensional face that is obtained by replacing each piece by its upper-right corner square -- in particular, the apex of any $0$--dimensional cell is that $0$-cell itself.  We use discrete Morse theory to collapse our cell complex so that among the cells remaining, at most one cell has any given apex.

\begin{figure}[h!]
\begin{center}
\begin{tikzpicture}[scale=.7]
\draw (0, 0)--(5, 0) (0, 1)--(5, 1) (0, 2)--(5, 2) (0, 3)--(5, 3) (0, 4)--(5, 4) (0, 5)--(5, 5) (0, 0)--(0, 5) (1, 0)--(1, 5) (2, 0)--(2, 5) (3, 0)--(3, 5) (4, 0)--(4, 5) (5, 0)--(5, 5);
\draw[draw=gray!40,fill=gray!40] 
(.1, .1)--(.9, .1)--(.9, 1.9)--(.1, 1.9)--cycle
(1.1, .1)--(1.9, .1)--(1.9, 1.9)--(1.1, 1.9)--cycle
(3.1, .1)--(4.9, .1)--(4.9, .9)--(3.1, .9)--cycle
(2.1, 1.1)--(3.9, 1.1)--(3.9, 1.9)--(2.1, 1.9)--cycle
(2.1, 2.1)--(2.9, 2.1)--(2.9, 2.9)--(2.1, 2.9)--cycle
(2.1, 3.1)--(3.9, 3.1)--(3.9, 4.9)--(2.1, 4.9)--cycle;
\end{tikzpicture}
\begin{tikzpicture}[scale=.7]
\draw[->] (2, 2.5)--(3, 2.5);
\node at (1.5, 0) {};
\node at (3.5, 0) {};
\end{tikzpicture}
\begin{tikzpicture}[scale=.7]
\draw (0, 0)--(5, 0) (0, 1)--(5, 1) (0, 2)--(5, 2) (0, 3)--(5, 3) (0, 4)--(5, 4) (0, 5)--(5, 5) (0, 0)--(0, 5) (1, 0)--(1, 5) (2, 0)--(2, 5) (3, 0)--(3, 5) (4, 0)--(4, 5) (5, 0)--(5, 5);
\draw[draw=gray!40,fill=gray!40] 
(.1, 1.1)--(.9, 1.1)--(.9, 1.9)--(.1, 1.9)--cycle
(1.1, 1.1)--(1.9, 1.1)--(1.9, 1.9)--(1.1, 1.9)--cycle
(4.1, .1)--(4.9, .1)--(4.9, .9)--(4.1, .9)--cycle
(3.1, 1.1)--(3.9, 1.1)--(3.9, 1.9)--(3.1, 1.9)--cycle
(2.1, 2.1)--(2.9, 2.1)--(2.9, 2.9)--(2.1, 2.9)--cycle
(3.1, 4.1)--(3.9, 4.1)--(3.9, 4.9)--(3.1, 4.9)--cycle;
\end{tikzpicture}
\end{center}
\caption{The apex of a rectangle arrangement replaces each piece by its upper-right corner.  The correspondence does not depend on the labels of the pieces, so the labels are not shown.}\label{apexfig}
\end{figure}

\begin{theorem}\label{mainthm}
There is a discrete gradient vector field on the cubical complex $\cell{n}{p}{q}$ with the following properties.
\begin{enumerate}
\item Every matched pair consists of two cells with the same apex.\label{apexprop}
\item Among the cells with a given apex, at most one cell is critical (unmatched).\label{oneprop}
\item The matching is $S_n$--equivariant: if cells $e_1$ and $e_2$ are a matched pair, and we apply the same permutation to the labels in the rectangle arrangements of $e_1$ and $e_2$, then the two resulting cells are also a matched pair. \label{snprop}
\end{enumerate}
\end{theorem}

The proof relies on constructing what we call the apex graph, which facilitates the enumeration of all the cells with a given apex.  In Lemmas~\ref{pathslemma} and~\ref{bijlemma} we prove the basic properties of the apex graph, and in Lemma~\ref{matchlemma} we define the matching for Theorem~\ref{mainthm} in the language of the apex graph.  After that, it is straightforward to finish the proof of Theorem~\ref{mainthm}.

Given any rectangle arrangement, we describe the locations of the pieces according to the coordinates of their upper-right corner squares, so that we say that a piece is at $(i, j)$ if its upper-right corner is in column $i$ (from left to right) and row $j$ (from bottom to top) of our $p$ by $q$ rectangle.  (Alternatively, the center of the upper-right board square has coordinates $(i, j)$ in the plane.)  Giving the coordinates of each piece is the same as specifying the apex of our cell.  To distinguish cells with the same apex, we need to specify, for each piece, whether it has height $1$ or $2$ and whether it has width $1$ or $2$.  Not all these possibilities give rise to valid rectangle arrangements, because some pieces may overlap or hang off the board.  For each possible apex, we construct the \textbf{\textit{apex graph}} to record these possible conflicts, as shown in Figure~\ref{apexgraphfig}.

\begin{figure}[h!]
\begin{center}
\begin{tikzpicture}[scale=.7]
\draw (0, 0)--(5, 0) (0, 1)--(5, 1) (0, 2)--(5, 2) (0, 3)--(5, 3) (0, 4)--(5, 4) (0, 5)--(5, 5) (0, 0)--(0, 5) (1, 0)--(1, 5) (2, 0)--(2, 5) (3, 0)--(3, 5) (4, 0)--(4, 5) (5, 0)--(5, 5);
\draw[draw=gray!40,fill=gray!40] 
(.1, 1.1)--(.9, 1.1)--(.9, 1.9)--(.1, 1.9)--cycle
(1.1, 1.1)--(1.9, 1.1)--(1.9, 1.9)--(1.1, 1.9)--cycle
(4.1, .1)--(4.9, .1)--(4.9, .9)--(4.1, .9)--cycle
(3.1, 1.1)--(3.9, 1.1)--(3.9, 1.9)--(3.1, 1.9)--cycle
(2.1, 2.1)--(2.9, 2.1)--(2.9, 2.9)--(2.1, 2.9)--cycle
(3.1, 4.1)--(3.9, 4.1)--(3.9, 4.9)--(3.1, 4.9)--cycle;

\draw[line width=0.5mm] (2, 2.5)--(3, 1.5) (3.5, 1)--(4, .5);
\draw[fill=white] 
(.5, 1) circle (.07) 
(1.5, 1) circle (.07)
(2, 2.5) circle (.07)
(2.5, 2) circle (.07)
(3, 1.5 ) circle (.07)
(3.5, 1) circle (.07)
(4, .5) circle (.07)
(3, 4.5) circle (.07)
(3.5, 4) circle (.07);
\end{tikzpicture}
\end{center}
\caption{To find the apex graph, we place one vertex for each direction that a piece in the apex can extend, and draw edges between directions where the pieces cannot extend simultaneously.}\label{apexgraphfig}
\end{figure}

The apex graph has at most two vertices per piece.  If our apex has a piece at $(i, j)$, then we let $(i-\frac{1}{2}, j)$ -- the center of the left edge of the $(i, j)$ board square -- be a vertex of the apex graph if and only if the piece at $(i, j)$ can have width $2$ in some cell with that apex -- that is, if $i>1$ and there is no piece at $(i-1, j)$.  Similarly, we let $(i, j-\frac{1}{2})$ -- the center of the lower edge of the $(i, j)$ board square -- be a vertex if and only if there is a piece at $(i, j)$, we have $j>1$, and there is no piece at $(i, j-1)$.

The edges of the apex graph record which of the width $2$ or height $2$ options would conflict with each other.  A piece at $(i, j)$ can have width $2$ or height $2$ but not both when there are no pieces at $(i-1, j)$ and $(i, j-1)$, but there is a piece at $(i-1, j-1)$.  In this case we draw an edge between the vertices $(i-\frac{1}{2}, j)$ and $(i, j-\frac{1}{2})$.  The other possible conflict is between pieces at $(i, j)$ and $(i-1, j+1)$.  If there is no piece at $(i-1, j)$, then the $(i, j)$ piece may have width $2$, and the $(i-1, j+1)$ piece may have height $2$, but not both simultaneously.  In this case we draw an edge between $(i-\frac{1}{2}, j)$ and $(i-1, j+\frac{1}{2})$.  These two types of edges give all the edges in the apex graph.

\begin{lemma}\label{pathslemma}
Each apex graph is a disjoint union of path graphs.
\end{lemma}

\begin{proof}
The two types of edges have the same slope and length when drawn on the coordinate lattice.  Any graph that can be drawn in this way is a disjoint union of paths.  Note that some of the paths may be single vertices.
\end{proof}

\begin{lemma}\label{bijlemma}
The set of cells with a given apex is in bijection with the set of independent sets in its apex graph.
One cell is a face of another if and only if the independent set corresponding to the first cell under this bijection is a subset of the independent set corresponding to the second cell.
\end{lemma}

\begin{figure}[h!]
\begin{center}
\begin{tikzpicture}[scale=.7]
\draw (0, 0)--(5, 0) (0, 1)--(5, 1) (0, 2)--(5, 2) (0, 3)--(5, 3) (0, 4)--(5, 4) (0, 5)--(5, 5) (0, 0)--(0, 5) (1, 0)--(1, 5) (2, 0)--(2, 5) (3, 0)--(3, 5) (4, 0)--(4, 5) (5, 0)--(5, 5);
\draw[draw=gray!40,fill=gray!40] 
(.1, .1)--(.9, .1)--(.9, 1.9)--(.1, 1.9)--cycle
(1.1, .1)--(1.9, .1)--(1.9, 1.9)--(1.1, 1.9)--cycle
(3.1, .1)--(4.9, .1)--(4.9, .9)--(3.1, .9)--cycle
(2.1, 1.1)--(3.9, 1.1)--(3.9, 1.9)--(2.1, 1.9)--cycle
(2.1, 2.1)--(2.9, 2.1)--(2.9, 2.9)--(2.1, 2.9)--cycle
(2.1, 3.1)--(3.9, 3.1)--(3.9, 4.9)--(2.1, 4.9)--cycle;

\draw[line width=0.5mm] (2, 2.5)--(3, 1.5) (3.5, 1)--(4, .5);
\draw[fill=black] 
(.5, 1) circle (.07) 
(1.5, 1) circle (.07)
(3, 1.5 ) circle (.07)
(4, .5) circle (.07)
(3, 4.5) circle (.07)
(3.5, 4) circle (.07);

\draw[fill=white] 
(2, 2.5) circle (.07)
(2.5, 2) circle (.07)
(3.5, 1) circle (.07);
\end{tikzpicture}
\end{center}
\caption{Cells with a given apex correspond to independent sets in the apex graph: we select the vertices corresponding to the directions where the apex pieces extend.  Here, the vertices in the independent set are drawn filled, and the other vertices are drawn empty.}\label{independentfig}
\end{figure}

\begin{proof}
Given a cell with a given apex, we find the corresponding subset of vertices in the apex graph by considering each piece in the associated rectangle arrangement, say at $(i, j)$, selecting vertex $(i-\frac{1}{2}, j)$ if the piece has width $2$, and selecting vertex $(i, j-\frac{1}{2})$ if it has height $2$, as in Figure~\ref{independentfig}.  The construction guarantees that these are in fact vertices of the apex graph and that no two of them share an edge.

For the converse, suppose that we have an independent set in the apex graph.  We select our pieces to have width $2$ and/or height $2$ according to which vertices are in the independent set, and we want to check whether the pieces overlap or hang off the board.  Consider the $(i, j)$ piece.  It cannot hang off the board or overlap with a piece at $(i-1, j)$ or $(i, j-1)$, because the vertices corresponding to those possibilities are not in the apex graph.  It cannot overlap with a piece at $(i-1, j-1)$ because that would mean choosing both vertices $(i-\frac{1}{2}, j)$ and $(i, j-\frac{1}{2})$ which would be adjacent.  And, it cannot overlap with a piece at $(i-1, j+1)$, because that would mean choosing both vertex $(i-\frac{1}{2}, j)$ and $(i-1, j+\frac{1}{2})$ which would be adjacent.  Symmetrically, by swapping the roles of the two pieces, we see that the piece at $(i, j)$ also cannot overlap with the pieces at $(i+1, j)$, $(i, j+1)$, $(i+1, j+1)$, or $(i+1, j-1)$.  This exhausts all the possibilities for how two pieces of width and height at most $2$ might overlap, and shows that we have a bijection.

For the second property, suppose that cells $e$ and $f$ have the same apex.  Then $f$ is a face of $e$ if and only if every piece of width $2$ in $f$ also has width $2$ in $e$, and every piece of height $2$ in $f$ also has height $2$ in $e$.  This is equivalent to the condition that the independent set corresponding to $f$ is a subset of the independent set corresponding to $e$.
\end{proof}

Thinking of the cells as independent sets in the apex graph suggests how to think about pairing them up.  The dimension of a cell is equal to the number of vertices in the independent set corresponding to that cell.  So, if cells $e$ and $f$ have the same apex, then $f$ is a face of $e$ with $\dim f = \dim e - 1$ if and only if the independent set of $f$ is a subset of the independent set of $e$ and the two sets differ by one vertex.  When two independent sets differ by one vertex, we say that they are \textit{\textbf{adjacent}}.

\begin{lemma}\label{matchlemma}
Given a disjoint union of paths, there is a matching on the set of independent sets such that every matched pair of independent sets are adjacent and at most one independent set is unmatched.
\end{lemma}

\begin{proof}
We start by proving the statement for one connected path of $k$ vertices.  We express the independent sets as binary strings of length $k$ with no consecutive $1$'s, so that $0$ indicates that the vertex in that position is not part of the independent set, and $1$ indicates that the vertex is part of the independent set.  The matching is defined recursively.  For $k=1$ the strings are $0$ and $1$, which we match as a pair.  For $k=2$ the strings are $00$, $01$, and $10$; we match $00$ with $10$ and leave $01$ unmatched.  For $k>2$, each string begins with $00$, $10$, or $010$.  We match the strings beginning with $00$ to the strings beginning with $10$ such that each matched pair differs only in the first bit.  Then, for the strings beginning with $010$ we ignore the first three bits and use the matching for the $k-3$ case.  

The result is that for $k\equiv 1\mod 3$, all strings are matched; for $k \equiv 0 \mod 3$, the only unmatched string consists of repeating copies of $010$; and for $k \equiv 2 \mod 3$ the only unmatched string consists of repeating copies of $010$ followed by $01$ at the end.  This proves the lemma for the case of one path.

\begin{figure}[h!]
\begin{center}
\begin{tikzpicture}
\draw (0, 0)--(3, 0) (0, 1)--(2, 1);
\draw[fill=black] (.5, 0) circle (.07)
(.5, 1) circle (.07)
(2.5, 0) circle (.07)
(2, 1) circle (.07);
\draw[fill=white] (0, 0) circle (.07)
(0, 1) circle (.07)
(1, 0) circle (.07)
(1, 1) circle (.07)
(1.5, 0) circle (.07)
(1.5, 1) circle (.07)
(2, 0) circle (.07)
(3, 0) circle (.07);

\node at (0, 0) [label=below:{$0$}] {};
\node at (.5, 0) [label=below:{$1$}] {};
\node at (1, 0) [label=below:{$0$}] {};
\node at (1.5, 0) [label=below:{$0$}] {};
\node at (2, 0) [label=below:{$0$}] {};
\node at (2.5, 0) [label=below:{$1$}] {};
\node at (3, 0) [label=below:{$0$}] {};

\node at (0, 1) [label=below:{$0$}] {};
\node at (.5, 1) [label=below:{$1$}] {};
\node at (1, 1) [label=below:{$0$}] {};
\node at (1.5, 1) [label=below:{$0$}] {};
\node at (2, 1) [label=below:{$1$}] {};
\end{tikzpicture}
\begin{tikzpicture}
\draw[<->] (-.5, .5)--(.5, .5);
\node at (-1, -.5) {};
\node at (1.5, -.5) {};
\end{tikzpicture}
\begin{tikzpicture}
\draw (0, 0)--(3, 0) (0, 1)--(2, 1);
\draw[fill=black] (.5, 0) circle (.07)
(.5, 1) circle (.07)
(1.5, 0) circle (.07)
(2.5, 0) circle (.07)
(2, 1) circle (.07);
\draw[fill=white] (0, 0) circle (.07)
(0, 1) circle (.07)
(1, 0) circle (.07)
(1, 1) circle (.07)
(1.5, 1) circle (.07)
(2, 0) circle (.07)
(3, 0) circle (.07);

\node at (0, 0) [label=below:{$0$}] {};
\node at (.5, 0) [label=below:{$1$}] {};
\node at (1, 0) [label=below:{$0$}] {};
\node at (1.5, 0) [label=below:{$1$}] {};
\node at (2, 0) [label=below:{$0$}] {};
\node at (2.5, 0) [label=below:{$1$}] {};
\node at (3, 0) [label=below:{$0$}] {};

\node at (0, 1) [label=below:{$0$}] {};
\node at (.5, 1) [label=below:{$1$}] {};
\node at (1, 1) [label=below:{$0$}] {};
\node at (1.5, 1) [label=below:{$0$}] {};
\node at (2, 1) [label=below:{$1$}] {};
\end{tikzpicture}
\end{center}
\caption{Given an independent set on a disjoint union of paths, to find its match we select the first component that is not critical, ignore any $010$ prefixes, and flip the first bit of the remainder.}\label{matchingfig}
\end{figure}

For several disjoint paths, we select some ordering on them.  Given an independent set, if its restriction to each path agrees with the unmatched independent set from the one-path case, we leave it unmatched.  Otherwise, we find the first path $P$ where this is not true.  To find the matching independent set, we keep all the other paths as they are and alter the set on $P$ to be the matching set from the one-path case, as in Figure~\ref{matchingfig}.  There is an unmatched independent set if and only if none of the paths has $1 \mod 3$ vertices, and in this case the unmatched set corresponds to repeating $010$ on each path.
\end{proof}

To finish the proof of Theorem~\ref{mainthm}, we need to check that the matching we have just defined determines a discrete gradient vector field with the properties we are looking for.

\begin{proof}[Proof of Theorem~\ref{mainthm}]
The discrete vector field is defined as follows.  Given a cell, we find its apex and the apex graph.  Encoding the original rectangle arrangement as an independent set in the apex graph (Lemma~\ref{bijlemma}), we find the matching independent set (Lemma~\ref{matchlemma}) if there is one, and decode to get another cell with the same apex.  Properties~(\ref{apexprop}) and~(\ref{snprop}) are automatic from the construction, and Property~(\ref{oneprop}) is a consequence of Lemma~\ref{matchlemma}.

We still need to check that the discrete vector field is gradient.  We want to show that there does not exist a cycle of cells $e_1, f_1, e_2, f_2, \ldots, e_r, f_r, e_{r+1}=e_1$ such that every $e_i$ and $f_i$ are a matched pair (where, in particular, $e_i$ is a face of $f_i$),     and every $f_i$ is a face of $e_{i+1}$ with $\dim f_i = \dim e_{i+1} - 1$.  We observe that because $f_i$ is a face of $e_{i+1}$, if the apex of $e_{i+1}$ is not equal to the apex of $f_i$, then it differs by moving some piece one square left or down.  Every pair $e_i$ and $f_i$ have the same apex, so as the sequence continues, the apex keeps moving leftward and downward, making it impossible to have a cycle unless all cells in it have the same apex.

Thus we may assume that the cells $e_1, f_1, \ldots, e_r, f_r$ all have the same apex.  We can encode these cells as independent sets in the apex graph.  To go from $e_1$ to $f_1$, we delete one vertex $v$ from the independent set of $e_1$, and to go from $f_1$ to $e_2$, we add one vertex $w$ to the independent set of $f_1$.  Remembering the ordering of the paths and vertices in the apex graph, we observe that up until $v$, the independent set for $f_1$ agrees with the unmatched independent set, so any added vertices there would destroy the property of being an independent set.  Thus $w$ cannot be at or before $v$.  If the vertex immediately after $v$ is on the same path, then $w$ can be that vertex.  But $w$ cannot be anywhere else after $v$, because if so, then the matching independent set to $e_2$, which we have supposed is $f_2$, would have $v$ added rather than a vertex subtracted -- it would have the wrong dimension -- giving a contradiction unless $w$ is immediately after $v$.

\begin{figure}[h!]
\begin{center}
\begin{tikzpicture}
\draw (0, 0)--(2, 0) (0, 1.5)--(2, 1.5) (0, 3)--(2, 3);
\draw[fill=white] (0, 0) circle (.07)
(1, 0) circle (.07)
(1.5, 0) circle (.07)
(0, 1.5) circle (.07)
(1, 1.5) circle (.07)
(1.5, 1.5) circle (.07)
(2, 1.5) circle (.07)
(0, 3) circle (.07)
(1, 3) circle (.07)
(2, 3) circle (.07);
\draw[fill=black] (.5, 0) circle (.07)
(2, 0) circle (.07)
(.5, 1.5) circle (.07)
(.5, 3) circle (.07)
(1.5, 3) circle (.07);
\node at (2, 0) [label=above:{$w$}] {};
\node at (1.5, 1.5) [label=above:{$v$}] {};
\node at (2, 1.5) [label=above:{$w$}] {};
\node at (1.5, 3) [label=above:{$v$}] {};
\draw[->] (1, 1)--(1, .5);
\draw[->] (1, 2.5)--(1, 2);
\node at (0, 0) [label=left:{\ldots}] {};
\node at (0, 1.5) [label=left:{\ldots}] {};
\node at (0, 3) [label=left:{\ldots}] {};
\node at (2, 0) [label=right:{\ldots}] {};
\node at (2, 1.5) [label=right:{\ldots}] {};
\node at (2, 3) [label=right:{\ldots}] {};
\node at (-1.5, 0) {$e_2$};
\node at (-1.5, 1.5) {$f_1$};
\node at (-1.5, 3) {$e_1$};
\end{tikzpicture}
\end{center}
\caption{In a sequence of cells with the same apex, alternating between two consecutive dimensions, with consecutive pairs alternating between matched and incident, the corresponding independent sets look more and more like the unmatched set, and thus cannot cycle.}\label{sequencefig}
\end{figure}

Thus we cannot have a cycle $e_1, f_1, \ldots, e_r, f_r, e_{r+1}=e_1$, because each successive item agrees more and more with the unmatched set, as in Figure~\ref{sequencefig}: the independent set of $e_1$ agrees before $v$, the independent set of $f_1$ agrees through $v$, the independent set of $e_2$ agrees through $w$, and so on.  So, our discrete vector field is gradient and has all three desired properties from the theorem statement.
\end{proof}

We prove one last theorem in this section, which helps with assessing the dimension of critical cells in the following section. For the following, we divide each unit square in  the $p$ by $q$ grid into two \emph{half-squares} by drawing a diagonal line from the upper-right to the lower-left corner.

\begin{theorem} \label{thm-halfsquares}
There is a function $r$ that assigns a set of half-squares to each vertex of the apex graph, with the following properties:
\begin{enumerate}
\item For any vertex $v$, the set $r(v)$ has four half-squares if $v$ is the only vertex of a path, three half-squares if $v$ is the first or last vertex of a path, and two half-squares otherwise.
\item The sets $r(v)$ are disjoint for all $v$.
\end{enumerate}
\end{theorem}




\begin{proof}[Proof of Theorem~\ref{thm-halfsquares}]
Recalling that we can draw each vertex of the apex graph as the midpoint of an edge between a square occupied by an apex piece and an unoccupied square, we set $r(v)$ to contain both of the half-squares that $v$ touches, as in Figure~\ref{fig-rule0}.  

\begin{figure}[h!]
\begin{center}
\begin{tikzpicture}[>=stealth]
\draw (0, 0)--(2, 0) (0, 1)--(2, 1) (0, 0)--(0, 1) (1, 0)--(1, 1) (2, 0)--(2, 1);
\draw[draw=gray!40,fill=gray!40] (1.1, .1)--(1.9, .1)--(1.9, .9)--(1.1, .9)--cycle;
\draw[<->] (.75, .75)--(1.25, .25);
\draw[fill=white] (1, .5) circle (.07);
\draw[line width=.5mm] (0, 0)--(1, 0)--(2, 1)--(1, 1)--cycle;
\end{tikzpicture}\hspace{10pt}
\begin{tikzpicture}[>=stealth]
\draw (0, 0)--(0, 2) (1, 0)--(1, 2) (0, 0)--(1, 0) (0, 1)--(1, 1) (0, 2)--(1, 2);
\draw[draw=gray!40,fill=gray!40] (.1, 1.1)--(.1, 1.9)--(.9, 1.9)--(.9, 1.1)--cycle;
\draw[<->] (.25, 1.25)--(.75, .75);
\draw[fill=white] (.5, 1) circle (.07);
\draw[line width=.5mm] (0, 0)--(1, 1)--(1, 2)--(0, 1)--cycle;
\end{tikzpicture}
\end{center}
\caption{Each vertex is assigned the two half-squares it touches.  If vertex $v$ has neighbors in both directions, then $r(v)$ contains only these two half-squares.}\label{fig-rule0}
\end{figure}

We think of the vertices as ordered first by the sum of coordinates and then by the column coordinate, so that the ordering starts in the lower-left corner and goes right and down along diagonals.  If $v$ is the first or last vertex of a path, we need to find another half-square to add to $r(v)$.  There are several cases, shown in Figure~\ref{fig-rule135}:
\begin{enumerate}
\item If $v$ is the first vertex of a path and is on a vertical edge:  Add in the remainder of the square to the left of $v$.
\item If $v$ is the last vertex of a path and is on a horizontal edge: Add in the remainder of the square below $v$.
\item If $v$ is the first vertex of a path and is on a horizontal edge, and there is no vertex on the preceding (above-left) edge:  Add in the remainder of the square above $v$.
\item If $v$ is the last vertex of a path and is on a vertical edge, and there is no vertex on the following (below-right) edge: Add in the remainder of the square to the right of $v$.
\item If $v$ is the first vertex of a path and is on a horizontal edge, and there is a (non-adjacent) vertex on the preceding edge: Add the half-square to the left of the square below $v$.
\item If $v$ is the last vertex of a path and is on a vertical edge, and there is a (non-adjacent) vertex on the following edge: Add the half-square below the square to the left of $v$.
\end{enumerate}

\begin{figure}[h!]
\begin{center}
\begin{tikzpicture}[>=stealth]
\draw (0, 0)--(2, 0) (0, 1)--(2, 1) (0, 0)--(0, 1) (1, 0)--(1, 1) (2, 0)--(2, 1);
\draw[draw=gray!40,fill=gray!40] (1.1, .1)--(1.9, .1)--(1.9, .9)--(1.1, .9)--cycle;
\draw[->] (1, .5)--(1.25, .25);
\draw[fill=white] (1, .5) circle (.07);
\draw[line width=.5mm] (0, 0)--(1, 0)--(2, 1)--(0, 1)--cycle;
\draw (.4, .9)--(.6, 1.1) (.4, 1.1)--(.6, .9);
\end{tikzpicture}\hspace{10pt}
\begin{tikzpicture}[>=stealth]
\draw (0, 0)--(0, 2) (1, 0)--(1, 2) (0, 0)--(1, 0) (0, 1)--(1, 1) (0, 2)--(1, 2);
\draw[draw=gray!40,fill=gray!40] (.1, 1.1)--(.1, 1.9)--(.9, 1.9)--(.9, 1.1)--cycle;
\draw[->] (.5, 1)--(.75, .75);
\draw[fill=white] (.5, 1) circle (.07);
\draw[line width=.5mm] (0, 0)--(1, 1)--(1, 2)--(0, 2)--cycle;
\draw (-.1, 1.4)--(.1, 1.6) (.1, 1.4)--(-.1, 1.6);
\end{tikzpicture}\hspace{10pt}
\begin{tikzpicture}[>=stealth]
\draw (0, 0)--(2, 0) (0, 1)--(2, 1) (0, 2)--(2, 2) (0, 0)--(0, 2) (1, 0)--(1, 2) (2, 0)--(2, 2);
\draw[draw=gray!40,fill=gray!40] (1.1, 1.1)--(1.9, 1.1)--(1.9, 1.9)--(1.1, 1.9)--cycle;
\draw[<-] (.75, 1.75)--(1, 1.5);
\draw[->] (1.5, 1)--(1.75, .75);
\draw[fill=white] (1, 1.5) circle (.07) (1.5, 1) circle (.07);
\draw[line width=.5mm] (1, 0)--(2, 1)--(2, 2)--(0, 0)--cycle (0, 1)--(0, 0)--(2, 2)--(1, 2)--cycle;
\end{tikzpicture}
\end{center}
\caption{For the first vertex of a path, rules (1), (3), and (5) specify how to add a third half-square; the $\times$ symbol indicates an absence of vertex.  The third picture also shows an instance of rule (6).  For the last vertex of a path, rules (2), (4), and (6) are analogous.}\label{fig-rule135}
\end{figure}

\begin{figure}[h!]
\begin{center}
\begin{tikzpicture}[>=stealth]
\node at (-.1, -.1) {};
\draw (0, 0)--(2, 0) (0, 1)--(2, 1) (0, 0)--(0, 1) (1, 0)--(1, 1) (2, 0)--(2, 1);
\draw[draw=gray!40,fill=gray!40] (1.1, .1)--(1.9, .1)--(1.9, .9)--(1.1, .9)--cycle;
\draw[fill=white] (1, .5) circle (.07);
\draw[line width=.5mm] (0, 0)--(2, 0)--(2, 1)--(0, 1)--cycle;
\draw (.4, .9)--(.6, 1.1) (.4, 1.1)--(.6, .9);
\draw (1.4, -.1)--(1.6, .1) (1.4, .1)--(1.6, -.1);
\end{tikzpicture}\hspace{10pt}
\begin{tikzpicture}[>=stealth]
\node at (-.1, -.1) {};
\draw (0, 0)--(2, 0) (0, 1)--(2, 1) (0, 2)--(2, 2) (0, 0)--(0, 2) (1, 0)--(1, 2) (2, 0)--(2, 2);
\draw[draw=gray!40,fill=gray!40] (1.1, 1.1)--(1.9, 1.1)--(1.9, 1.9)--(1.1, 1.9)--cycle;
\draw[->] (1.5, 1)--(1.75, .75);
\draw[fill=white] (1, 1.5) circle (.07) (1.5, 1) circle (.07);
\draw[line width=.5mm] (0, 0)--(2, 2)--(0, 2)--cycle;
\draw (.4, 1.9)--(.6, 2.1) (.4, 2.1)--(.6, 1.9);
\end{tikzpicture}
\end{center}
\caption{If the only vertex of a path is on a vertical edge, rules (1) and (4) or rules (1) and (6) assign four half-squares to that vertex.  If the vertex is on a horizontal edge, rules (2) and (3) or rules (2) and (5) are analogous.}\label{fig-rule146}
\end{figure}

In the case where $v$ is the only vertex of the path, if $v$ is on a vertical edge, then either rules (1) and (4) or rules (1) and (6) apply, as shown in Figure~\ref{fig-rule146}, and if $v$ is on a horizontal edge, then either rules (2) and (3) or rules (2) and (5) apply, so that $r(v)$ has four half-squares in total.  This completes the definition of $r(v)$.

We need to check that no half-square has been assigned twice.  To do this, we consider the assignment from the point of view of each square.  Consider a square that is occupied by a piece in the apex arrangement.  It may have vertices on its left or lower edges.  If it has both vertices, then half of the square is assigned to each vertex.  If it has one vertex, then all of the square is assigned to that vertex, by rule (3) or rule (4).  If it has no vertex, then none of the rules assign that square to any vertex.

Similarly, consider a square that is unoccupied in the apex arrangement.  It may have vertices on its right or upper edges.  If it has both vertices, then half of the square is assigned to each vertex.  If it has one vertex, then all of the square is assigned to that vertex, by rule (1) or rule (2).  (Note that rule (1) and rule (2) cannot apply to an unoccupied square with both vertices, because in this case the two vertices would be adjacent.)  If our unoccupied square has no vertices, then we divide the square in half.  The lower-right half gets assigned by rule (5) to the same vertex (if any) as the half-square to its right, and the upper-left half gets assigned by rule (6) to the same vertex (if any) as the half-square above it.

In each case, only one rule can apply to each half-square, so each half-square can be assigned to only one vertex.
\end{proof}



\section{Homology-vanishing theorems}

The existence of the cell complex $\cell{n}{p}{q}$ and the discrete gradient on it allow us to establish a number of homology-vanishing results.

\begin{theorem} \label{thm:hv:pq-n}
If $j > pq - n$, then $H_j [ \config{n}{p}{q} ] = 0$.
\end{theorem}

\begin{proof}[Proof of Theorem \ref{thm:hv:pq-n}] This is almost immediate from the homotopy equivalence $\config{n}{p}{q} \sim \cell{n}{p}{q}$.  Consider the dimensions of the cells in $\cell{n}{p}{q}$. A cell is indexed by a collection of $n$ non-overlapping rectangular pieces in a $p \times q$ grid. A $1 \times 1$ piece contributes $0$ to the dimension of the cell, a $1 \times 2$ or $2 \times 1$ piece contributes $1$, and a $2 \times 2$ piece contributes $2$. The total area of the pieces is at most $pq$. So the largest dimension of a cell is at most $pq-n$. By the definition of cellular homology, there is no homology above the dimension of the cell complex itself, so $H_j [ \config{n}{p}{q} ] = 0$ for $j > pq - n$.
\end{proof}

\begin{theorem} \label{thm:hv:n}
If $j > n$, then $H_j [ \config{n}{p}{q} ] = 0$.
\end{theorem}

\begin{proof}[Proof of Theorem \ref{thm:hv:n}]
This follows from the properties of the discrete gradient described in Section \ref{sec:dMt}. Every cell is indexed by a collection of non-overlapping rectangular pieces, each piece either $1 \times 1$, $1 \times 2$, $2 \times 1$, or $2 \times 2$. The analysis of the gradient shows that there are no critical cells indexed by a collection of pieces including a $2 \times 2$ piece. (When such a cell is encoded as an independent set in the apex graph as in Lemma~\ref{bijlemma}, the $2 \times 2$ piece corresponds to two vertices of the independent set that are consecutive but not adjacent.  However, the proof of Lemma~\ref{matchlemma} implies that in a critical cell, the first vertex of each path is never part of the corresponding independent set.) Hence the dimension of a critical cell is at most $n$.
\end{proof}

\begin{theorem} \label{thm:pq3}
If $j > pq / 3$, then $H_j [ \config{n}{p}{q} ] = 0$.
\end{theorem}

\begin{proof}[Proof of Theorem \ref{thm:pq3}]
This follows from Theorem \ref{thm-halfsquares}. A critical cell corresponds to an independent set that on each path looks like $010\ldots010$ or $010\ldots01$ depending on whether the number of vertices in the path is $0$ or $2$ mod $3$.  The dimension of the critical cell is the number of vertices in the independent set.  If the path has $k$ vertices, then the independent set has $k/3$ vertices in the first case, and has $(k+1)/3$ vertices in the second case.

For a path of $k$ vertices, the theorem allocates half-squares with a total area of $k+1$ to the vertices of the path.  The independent set for that path contributes at most $(k+1)/3$ to the dimension of the critical cell.  Thus, in total, the dimension of the critical cell is at most one third of the total area allocated to all the paths in the apex graph, and thus is at most $pq/3$.
\end{proof}

Putting together Theorems \ref{thm:hv:pq-n}, \ref{thm:hv:n}, and \ref{thm:pq3}, we have proved Theorem \ref{thm:homologyvanishing}.

\vanishinghomology*

\section{Nontrivial homology} \label{sec:nontrivial}

In this section, our main aim is to prove Theorem \ref{thm:allnonvanish}. We give several explicit constructions of nontrivial cycles, and then a method for interpolating between parameters.

\begin{lemma}  \label{lem:14attainable}
The points $(1/2,1/4)$ and $(3/4,1/4)$ are attainable.
\end{lemma}

\begin{proof}
Figure \ref{fig:222} shows a cycle in $H_1[\config{2}{2}{2}]$. More precisely, the figure illustrates a piecewise-linear map $i: S^1 \to \config{2}{2}{2}$, where we linearly interpolate at constant speed between the positions shown. Then if $[\sigma]$ is a generator of $H_1(S^1)$, the cycle we are describing is the image $i_*([\sigma])$. 

To show that this cycle is nontrivial, consider the map $f: \config{2}{2}{2} \to S^1$, where one takes the angle the line from the center of square 1 to the center of square 2 makes with the $x$-axis. In other words, define
$$f(x_1, y_1, x_2, y_2) = \frac{1}{\sqrt{(x_2-x_1)^2 + (y_2-y_1)^2}}(x_2-x_1,y_2-y_1).$$
The composition $f \circ i$ is a degree-one map $S^1 \to S^1$, and in particular the induced map $(f \circ i)_*$ is an isomorphism on $H_1$. So then $i$ must be injective on $H_1$.

Similarly, Figure \ref{fig:322} shows a cycle in $H_1[\config{3}{2}{2}]$. The figure illustrates a piecewise-linear map $i: S^1 \to \config{3}{2}{2}$, and the cycle we are interested in is the image. This also represents a nontrivial cycle in $H_1[\config{3}{2}{2}]$. Indeed, we have a natural projection map to $\config{2}{2}{2}$ where one forgets the coordinates of the third square, and then the argument above shows that the image of the cycle is still nontrivial in this projection.
\input{cycles.txt}
\input{cycles2.txt}
\end{proof}


The following lemma will be superseded later in this section by a stronger result, but we present the lemma and proof as a warmup, and we will also reuse the main construction in its proof later.

\begin{lemma} \label{lem:diag}
The point
$$(x,y) = \left( \frac{k}{k^2},\frac{k-1}{k^2} \right)= \left( \frac{1}{k}, \frac{1}{k} - \frac{1}{k^2} \right)$$
is attainable, for every $k \ge 1$.
\end{lemma}

\input{cycles3.txt}

\begin{proof}
Consider first Figure \ref{fig:2and1}. We illustrate a $2 \times 2$ square and $1 \times 1$ square orbiting each other in a $3 \times 3$ grid, as in Figure \ref{fig:222}. We can then put two $1 \times 1$ squares inside the $2 \times 2$ square, and these can orbit each other independently. So, together these motions describe a map $i: T^2 \to \config{3}{3}{3}$. This map is illustrated in Figure  \ref{fig:2333} and \ref{fig:2torus}. By induction, we can embed an $(k-1)$-dimensional torus realizing a nontrivial cycle in $\config{k}{k}{k}$, for every $k \ge 2$.
\begin{figure}
\centering
\begin{tikzpicture}[line width=0.25mm,scale=1]
\draw (0,0)--(3,0)--(3,3)--(0,3)--cycle;
\draw[fill=lightgray] (0,3)--++(1,0)--++(0,-1)--++(-1,0)--cycle;
\draw[fill=gray] (1,2)--++(1,0)--++(0,-1)--++(-1,0)--cycle;
\draw[fill=black] (2,1)--++(1,0)--++(0,-1)--++(-1,0)--cycle;
\draw (1,2) circle (0.71);
\draw (1.75,1.25) circle (1.07);
\draw [blue,line width=0.5mm] (0,3)--++(2,0)--++(0,-2)--++(-2,0)--cycle;
\end{tikzpicture}
\caption{A map $T^2 \to \config{3}{3}{3}$. The light gray and dark gray squares orbit each other inside the blue $2 \times 2$ square as in Figure \ref{fig:222}, while the black and blue square orbit each other independently as in Figure \ref{fig:2and1}.}
\label{fig:2333}
\end{figure}

\input{torus-map.txt} 

The same argument as before gives that this represents a nontrivial class in $H_2[ \config{3}{3}{3}]$. Indeed, compose with a map $f: \config{3}{3}{3} \to T^2 = S^1 \times S^1$ which assigns to the first coordinate the angle between the line segment from the center of square $1$ to the center of square $2$ and the $x$-axis. Similarly, the map assigns to the second coordinate the angle between the line segment from the center of square $1$ to the center of square $3$ and the $x$-axis. This is a degree one map $T^2 \to T^2$. The induced map $(f \circ i)_*$ is an isomorphism on homology, and so $i_*$ is injective.
\end{proof}

\begin{figure}
\centering

\begin{tikzpicture}[line width=0.25mm,scale=1]
\draw (0,0)--(4,0)--(4,4)--(0,4)--cycle;
\draw[fill=lightgray] (0,4)--(1,4)--(1,3)--(0,3)--cycle;
\draw[fill=gray] (1,3)--(2,3)--(2,2)--(1,2)--cycle;
\draw[fill=lightgray] (2,4)--(3,4)--(3,3)--(2,3)--cycle;
\draw[fill=gray] (3,3)--(4,3)--(4,2)--(3,2)--cycle;
\draw[fill=lightgray] (0,2)--(1,2)--(1,1)--(0,1)--cycle;
\draw[fill=gray] (1,1)--(2,1)--(2,0)--(1,0)--cycle;
\draw (1,3) circle (0.71);
\draw (3,3) circle (0.71);
\draw (1,1) circle (0.71);
\draw (2,2) circle (1.414);

\end{tikzpicture}
\caption{A map $T^4 \to \config{6}{4}{4}$. The three pairs of squares orbit each other in the blue $2 \times 2$ squares, and these three blue $2 \times 2$ squares orbit each other in the $4 \times 4$ square, as in Figure \ref{fig:322}. The image of this map gives a nontrivial cycle in $H_4 [ \config{6}{4}{4}]$, realizing the point $(x,y) = (3/8, 1/4)$.}
\label{fig:4644}
\end{figure}

Lemma \ref{lem:diag} shows that there are infinitely many points realized on the parabola $y = x - x^2$. The following lemma improves on that result, showing that there are infinitely many points realized on the parabola $y = x - (8/9) x^2$.

\begin{lemma}
The point 
$$(x,y) = \left( \frac{3}{4k},\frac{3}{4k}- \frac{1}{2k^2} \right)$$
is attainable for every $k \ge 1$.
\end{lemma}

\begin{proof}
The case $k = 1$ is already covered by \Cref{lem:14attainable}.
For any $k \ge 2$, we can embed a $(k-1)$-dimensional torus in a $\config{k}{k}{k}$. Now consider the configuration space $\config{3k}{2k}{2k}$. We can divide the $2k \times 2k$ grid into four $k \times k$ grids. Inside each, we use $k$ squares to embed an $(k-1)$ torus as in the proof of \Cref{lem:diag}. This describes a $(3k-3)$-torus, and the three $k \times k$ squares can orbit each other in the $2k \times 2k$ grid, giving one more dimension. So putting it all together, we have a $(3k-2)$-torus. This realizes the point
\[
(x,y) = \left( \frac{3k}{4k^2}, \frac{3k-2}{4k^2} \right)=
\left( \frac{3}{4k},\frac{3}{4k}- \frac{1}{2k^2} \right).
\]
The case $k=2$ is illustrated in Figure~\ref{fig:4644}, and the case $k=3$ in Figure~\ref{fig:7966}.
\end{proof}

\begin{lemma}
The point $( x,0)$
is attainable for every rational $x$ with $0 \le x \le 1$.
\end{lemma}

\begin{proof}
Indeed, suppose $x$ is a rational point in $[0,1]$, and write $x = a/b$, where $a$ is a non-negative integer, $b$ is a positive integer, and $a \le b$. Set $n = ab$ and $p = q= b$. By assumption, we have $a \le b$, so $n \le pq $ and the configuration space $\config{n}{p}{q}$ is nonempty, so $H_0[ \config{n}{p}{q} ] \neq 0$.
\end{proof}

\begin{figure}
\centering

\begin{tikzpicture}[line width=0.25mm,scale=2/3]
\draw (0,0)--(6,0)--(6,6)--(0,6)--cycle;
\draw[fill=lightgray] (0,6)--++(1,0)--++(0,-1)--++(-1,0)--cycle;
\draw[fill=gray] (1,5)--++(1,0)--++(0,-1)--++(-1,0)--cycle;
\draw[fill=black] (2,4)--++(1,0)--++(0,-1)--++(-1,0)--cycle;
\draw[fill=lightgray] (3,6)--++(1,0)--++(0,-1)--++(-1,0)--cycle;
\draw[fill=gray] (4,5)--++(1,0)--++(0,-1)--++(-1,0)--cycle;
\draw[fill=black] (5,4)--++(1,0)--++(0,-1)--++(-1,0)--cycle;
\draw[fill=lightgray] (0,3)--++(1,0)--++(0,-1)--++(-1,0)--cycle;
\draw[fill=gray] (1,2)--++(1,0)--++(0,-1)--++(-1,0)--cycle;
\draw[fill=black] (2,1)--++(1,0)--++(0,-1)--++(-1,0)--cycle;
\draw (1,5) circle (0.71);
\draw (4,5) circle (0.71);
\draw (1,2) circle (0.71);
\draw (1.75,4.25) circle (1.07);
\draw (4.75,4.25) circle (1.07);
\draw (1.75,1.25) circle (1.07);
\draw (3,3) circle (2.12);


\end{tikzpicture}
\caption{A map $T^7 \to \config{9}{6}{6}$, realizing the point $(x,y)=(1/4,7/36)$.}
\label{fig:7966}
\end{figure}

Finally, we show that we can rationally interpolate between all the points we have described.
Let $S$ be the set of points $$S = \left\{ \left( \frac{3}{4k},\frac{3}{4k}- \frac{1}{2k^2} \right) \,\middle\vert\,  k \ge 1 \right\}.$$
Let $I$ be the closed interval
$$ I = \{ (x,y) \mid 0 \le  x \le 1 \mathrm{\ and\ } y \ge 0 \}.$$ 

\nonvanish*

\begin{proof}
By Cartheodory's theorem, if $(r_1,r_2)$ is in the convex hull of $S \cup I$, then $(r_1,r_2)$ is in the convex hull of three points of $S \cup I$. Write $(r_1,r_2)$ as a rational convex combination of these three points, i.e.,
\[
(r_1,r_2) = \lambda_1 (u_1, v_1) + \lambda_2 (u_2, v_2) + \lambda_3 (u_3,v_3)
\]
with
\begin{enumerate}
    \item $(u_1,v_1),(u_2,v_2),(u_3,v_3) \in S \cup I$,
    \item $0 \le \lambda_1, \lambda_2, \lambda_3 \le 1$ with $\lambda_1 + \lambda_2 + \lambda_3 = 1$, and
    \item $\lambda_1, \lambda_2, \lambda_3$ all rational.
\end{enumerate}
By the previous lemmas, $(u_i,v_i)$ is realizable as a nontrivial homology class for hard squares in a square for $i=1, 2, 3$. Let $n_i,p_i, j_i$ be such that $u_i = n_i /p_i^2$ and $v_i = j_i / p_i^2$ for $i=1,2,3$.
Let $\lambda_i = a_i / b_i$ for $i=1,2,3$.
Set 
\begin{align*}
P &= p_1p_2p_3,\\
B &= b_1b_2b_3,\\
R &= PB,
\end{align*}
then let 
\[N = r_1 R^2 ,\]
and \[J = r_2 R^2.\]
If we can find a nontrivial class in 
\[H_J [ \config{N}{R}{R}],\]
we are done.

Partition the $R \times R$ square into $B^2$ smaller squares, each of dimension $P \times P$.
In a $\lambda_1$ fraction of these smaller squares (i.e., in $\lambda_1 B^2 = a_1 b_1 b_2^2 b_3^2$ of them), we realize $(u_1,v_1)$ as follows. Further partition each $P \times P$ square into $p_2^2 p_3^2$ squares, of dimensions $p_1 \times p_1$. In each of these squares, we can place $n_1$ squares and can then describe a map from a torus giving a nontrivial class in $H_{j_1} [ \config{n_1}{p_1}{p_1}]$. So in total, we place 
\[ (a_1 b_1 b_2^2 b_3^2)(p_2^2 p_3^2)n_1 = (\lambda_1 B^2) (P^2 / p_1^2) n_1 = \lambda_1 (n_1 / p_1^2) (P^2 B^2) = \lambda_1 u_1 R^2
\]
squares, and get a map from the torus of dimension 
\[(a_1 b_1 b_2^2 b_3^2)(p_2^2 p_3^2)j_1
= (\lambda_1 B^2) (P^2 / p_1^2) j_1 = \lambda_1 (j_1 / p_1^2) (P^2 B^2) = \lambda_1 v_1 R^2.
\]

Similarly, in a $\lambda_2$ fraction of these $P \times P$ squares we can realize $(u_2, v_2)$, by dividing up into $p_1^2 p_3^2$ smaller squares of dimension $p_2 \times p_2$, and in a $\lambda_3$ fraction of the $P \times P$ squares we realize $(u_3, v_3)$. 

Altogether, we have used
\[\lambda_1 u_1 R^2 + \lambda_2 u_2 R^2 + \lambda_3 u_3 R^2 = r_1R^2 = N\]
squares, and defined an embedded torus of dimension
\[\lambda_1 v_1 R^2 + \lambda_2 v_2 R^2 + \lambda_3 v_3 R^2 = r_2R^2 = J.\]
This describes a cycle in 
\[H_J [ \config{N}{P}{P}],\]
as desired. The cycle is nontrivial, as before -- we can compose with a map to $T^j$ such that the composed map $T^j \to T^j$ has degree one.
\end{proof}

\input{computations}

\bibliographystyle{plain}
\bibliography{refs}


\end{document}

%% file: attained.txt
\begin{figure}
\begin{tikzpicture}[scale = 10]
\fill [opacity = 0.1] (0,0)--(1/3,1/3)--(2/3,1/3)--(1,0)--cycle;
\draw [blue, line width = 0.3mm](0,0)--(1,0)--(3/4,1/4)--( 3/8 , 1/4 )--( 1/4 , 7/36 )--( 3/16 , 5/32 )--( 3/20 , 13/100 )--( 1/8 , 1/9 )--( 3/28 , 19/196 )--( 3/32 , 11/128 )--( 1/12 , 25/324 )--( 3/40 , 7/100 )--( 3/44 , 31/484 )--
( 1/16 , 17/288 )--( 3/52 , 37/676 )--( 3/56 , 5/98 )--( 1/20 , 43/900 )--cycle;
\draw[densely dashed] plot[domain=0:9/8] (\x, {\x-(8/9)*(\x)^2});
\draw [densely dashed] (0,1/4)--(1,1/4);
\draw (0,0)--(0.7,0.7);
\node at (0.75,0.75) {$y=x$};
\draw (1,0)--(0.3,0.7);
\node at (0.25,0.75) {$y=1-x$};
\draw (0,1/3)--(1,1/3);
\node at (1.1,1/3) {$y=1/3$};
\node at (1.1,1/4) {$y=1/4$};
\node at (1.1,1/7) {$y=x-(8/9)x^2$};
\draw [line width =0.5mm,->] (0,0)--(1.25,0);
\draw [line width =0.5mm,->] (0,0)--(0,0.75);
\node at (0.625,-0.03) {$\mathbf{x=n/pq}$};
\node[rotate=90] at (-0.03,0.375) {$\mathbf{y=j/pq}$};
\fill [blue] ( 0 , 0 ) circle[radius=0.2pt];
\fill [blue] ( 1 , 0 ) circle[radius=0.2pt];
\fill [blue] ( 1/2 , 0 ) circle[radius=0.2pt];
\fill [blue] ( 1/4 , 0 ) circle[radius=0.2pt];
\fill [blue] ( 1/6 , 0 ) circle[radius=0.2pt];
\fill [blue] ( 1/8 , 0 ) circle[radius=0.2pt];
\fill [blue] ( 1/10 , 0 ) circle[radius=0.2pt];
\fill [blue] ( 1/18 , 0 ) circle[radius=0.2pt];
\fill [blue] ( 1/24 , 0 ) circle[radius=0.2pt];
\fill [blue] ( 1/25 , 0 ) circle[radius=0.2pt];
\fill [blue] ( 1/36 , 0 ) circle[radius=0.2pt];
\fill [blue] ( 1/2 , 0 ) circle[radius=0.2pt];
\fill [blue] ( 1/2 , 1/4 ) circle[radius=0.2pt];
\fill [blue] ( 1/3 , 0 ) circle[radius=0.2pt];
\fill [blue] ( 1/3 , 1/6 ) circle[radius=0.2pt];
\fill [blue] ( 2/9 , 0 ) circle[radius=0.2pt];
\fill [blue] ( 2/9 , 1/9 ) circle[radius=0.2pt];
\fill [blue] ( 1/6 , 0 ) circle[radius=0.2pt];
\fill [blue] ( 1/6 , 1/12 ) circle[radius=0.2pt];
\fill [blue] ( 1/8 , 0 ) circle[radius=0.2pt];
\fill [blue] ( 1/8 , 1/16 ) circle[radius=0.2pt];
\fill [blue] ( 2/25 , 0 ) circle[radius=0.2pt];
\fill [blue] ( 2/25 , 1/25 ) circle[radius=0.2pt];
\fill [blue] ( 1/18 , 0 ) circle[radius=0.2pt];
\fill [blue] ( 1/18 , 1/36 ) circle[radius=0.2pt];
\fill [blue] ( 3/4 , 0 ) circle[radius=0.2pt];
\fill [blue] ( 3/4 , 1/4 ) circle[radius=0.2pt];
\fill [blue] ( 1/2 , 0 ) circle[radius=0.2pt];
\fill [blue] ( 1/2 , 1/6 ) circle[radius=0.2pt];
\fill [blue] ( 1/3 , 0 ) circle[radius=0.2pt];
\fill [blue] ( 1/3 , 1/9 ) circle[radius=0.2pt];
\fill [blue] ( 1/3 , 2/9 ) circle[radius=0.2pt];
\fill [blue] ( 3/16 , 0 ) circle[radius=0.2pt];
\fill [blue] ( 3/16 , 1/16 ) circle[radius=0.2pt];
\fill [blue] ( 3/16 , 1/8 ) circle[radius=0.2pt];
\fill [blue] ( 3/25 , 0 ) circle[radius=0.2pt];
\fill [blue] ( 3/25 , 1/25 ) circle[radius=0.2pt];
\fill [blue] ( 3/25 , 2/25 ) circle[radius=0.2pt];
\fill [blue] ( 1/12 , 0 ) circle[radius=0.2pt];
\fill [blue] ( 1/12 , 1/36 ) circle[radius=0.2pt];
\fill [blue] ( 1/12 , 1/18 ) circle[radius=0.2pt];
\fill [blue] ( 1 , 0 ) circle[radius=0.2pt];
\fill [blue] ( 2/3 , 0 ) circle[radius=0.2pt];
\fill [blue] ( 2/3 , 1/6 ) circle[radius=0.2pt];
\fill [blue] ( 1/2 , 0 ) circle[radius=0.2pt];
\fill [blue] ( 1/2 , 1/8 ) circle[radius=0.2pt];
\fill [blue] ( 1/2 , 1/4 ) circle[radius=0.2pt];
\fill [blue] ( 4/9 , 0 ) circle[radius=0.2pt];
\fill [blue] ( 4/9 , 1/9 ) circle[radius=0.2pt];
\fill [blue] ( 4/9 , 2/9 ) circle[radius=0.2pt];
\fill [blue] ( 1/3 , 0 ) circle[radius=0.2pt];
\fill [blue] ( 1/3 , 1/12 ) circle[radius=0.2pt];
\fill [blue] ( 1/3 , 1/6 ) circle[radius=0.2pt];
\fill [blue] ( 1/4 , 0 ) circle[radius=0.2pt];
\fill [blue] ( 1/4 , 1/16 ) circle[radius=0.2pt];
\fill [blue] ( 1/4 , 1/8 ) circle[radius=0.2pt];
\fill [blue] ( 1/4 , 3/16 ) circle[radius=0.2pt];
\fill [blue] ( 4/25 , 0 ) circle[radius=0.2pt];
\fill [blue] ( 4/25 , 1/25 ) circle[radius=0.2pt];
\fill [blue] ( 4/25 , 2/25 ) circle[radius=0.2pt];
\fill [blue] ( 4/25 , 3/25 ) circle[radius=0.2pt];
\fill [blue] ( 1/9 , 0 ) circle[radius=0.2pt];
\fill [blue] ( 1/9 , 1/36 ) circle[radius=0.2pt];
\fill [blue] ( 1/9 , 1/18 ) circle[radius=0.2pt];
\fill [blue] ( 1/9 , 1/12 ) circle[radius=0.2pt];
\fill [blue] ( 5/6 , 0 ) circle[radius=0.2pt];
\fill [blue] ( 5/6 , 1/6 ) circle[radius=0.2pt];
\fill [blue] ( 5/8 , 0 ) circle[radius=0.2pt];
\fill [blue] ( 5/8 , 1/8 ) circle[radius=0.2pt];
\fill [blue] ( 5/8 , 1/4 ) circle[radius=0.2pt];
\fill [blue] ( 1/2 , 0 ) circle[radius=0.2pt];
\fill [blue] ( 1/2 , 1/10 ) circle[radius=0.2pt];
\fill [blue] ( 1/2 , 1/5 ) circle[radius=0.2pt];
\fill [blue] ( 5/9 , 0 ) circle[radius=0.2pt];
\fill [blue] ( 5/9 , 1/9 ) circle[radius=0.2pt];
\fill [blue] ( 5/9 , 2/9 ) circle[radius=0.2pt];
\fill [blue] ( 5/12 , 0 ) circle[radius=0.2pt];
\fill [blue] ( 5/12 , 1/12 ) circle[radius=0.2pt];
\fill [blue] ( 5/12 , 1/6 ) circle[radius=0.2pt];
\fill [blue] ( 1/3 , 0 ) circle[radius=0.2pt];
\fill [blue] ( 1/3 , 1/15 ) circle[radius=0.2pt];
\fill [blue] ( 1/3 , 2/15 ) circle[radius=0.2pt];
\fill [blue] ( 1/3 , 1/5 ) circle[radius=0.2pt];
\fill [blue] ( 5/16 , 0 ) circle[radius=0.2pt];
\fill [blue] ( 5/16 , 1/16 ) circle[radius=0.2pt];
\fill [blue] ( 5/16 , 1/8 ) circle[radius=0.2pt];
\fill [blue] ( 5/16 , 3/16 ) circle[radius=0.2pt];
\fill [blue] ( 1/4 , 0 ) circle[radius=0.2pt];
\fill [blue] ( 1/4 , 1/20 ) circle[radius=0.2pt];
\fill [blue] ( 1/4 , 1/10 ) circle[radius=0.2pt];
\fill [blue] ( 1/4 , 3/20 ) circle[radius=0.2pt];
\fill [blue] ( 1/5 , 0 ) circle[radius=0.2pt];
\fill [blue] ( 1/5 , 1/25 ) circle[radius=0.2pt];
\fill [blue] ( 1/5 , 2/25 ) circle[radius=0.2pt];
\fill [blue] ( 1/5 , 3/25 ) circle[radius=0.2pt];
\fill [blue] ( 1/5 , 4/25 ) circle[radius=0.2pt];
\fill [blue] ( 1 , 0 ) circle[radius=0.2pt];
\fill [blue] ( 3/4 , 0 ) circle[radius=0.2pt];
\fill [blue] ( 3/4 , 1/8 ) circle[radius=0.2pt];
\fill [blue] ( 3/4 , 1/4 ) circle[radius=0.2pt];
\fill [blue] ( 3/5 , 0 ) circle[radius=0.2pt];
\fill [blue] ( 3/5 , 1/10 ) circle[radius=0.2pt];
\fill [blue] ( 3/5 , 1/5 ) circle[radius=0.2pt];
\fill [blue] ( 1/2 , 0 ) circle[radius=0.2pt];
\fill [blue] ( 1/2 , 1/12 ) circle[radius=0.2pt];
\fill [blue] ( 1/2 , 1/6 ) circle[radius=0.2pt];
\fill [blue] ( 1/2 , 1/4 ) circle[radius=0.2pt];
\fill [blue] ( 2/3 , 0 ) circle[radius=0.2pt];
\fill [blue] ( 2/3 , 1/9 ) circle[radius=0.2pt];
\fill [blue] ( 2/3 , 2/9 ) circle[radius=0.2pt];
\fill [blue] ( 1/2 , 0 ) circle[radius=0.2pt];
\fill [blue] ( 1/2 , 1/12 ) circle[radius=0.2pt];
\fill [blue] ( 1/2 , 1/6 ) circle[radius=0.2pt];
\fill [blue] ( 2/5 , 0 ) circle[radius=0.2pt];
\fill [blue] ( 2/5 , 1/15 ) circle[radius=0.2pt];
\fill [blue] ( 2/5 , 2/15 ) circle[radius=0.2pt];
\fill [blue] ( 2/5 , 1/5 ) circle[radius=0.2pt];
\fill [blue] ( 1/3 , 0 ) circle[radius=0.2pt];
\fill [blue] ( 1/3 , 1/18 ) circle[radius=0.2pt];
\fill [blue] ( 1/3 , 1/9 ) circle[radius=0.2pt];
\fill [blue] ( 1/3 , 1/6 ) circle[radius=0.2pt];
\fill [blue] ( 1/3 , 2/9 ) circle[radius=0.2pt];
\fill [blue] ( 3/8 , 0 ) circle[radius=0.2pt];
\fill [blue] ( 3/8 , 1/16 ) circle[radius=0.2pt];
\fill [blue] ( 3/8 , 1/8 ) circle[radius=0.2pt];
\fill [blue] ( 3/8 , 3/16 ) circle[radius=0.2pt];
\fill [blue] ( 3/8 , 1/4 ) circle[radius=0.2pt];
\fill [blue] ( 3/10 , 0 ) circle[radius=0.2pt];
\fill [blue] ( 3/10 , 1/20 ) circle[radius=0.2pt];
\fill [blue] ( 3/10 , 1/10 ) circle[radius=0.2pt];
\fill [blue] ( 3/10 , 3/20 ) circle[radius=0.2pt];
\fill [blue] ( 3/10 , 1/5 ) circle[radius=0.2pt];
\fill [blue] ( 6/25 , 0 ) circle[radius=0.2pt];
\fill [blue] ( 6/25 , 1/25 ) circle[radius=0.2pt];
\fill [blue] ( 6/25 , 2/25 ) circle[radius=0.2pt];
\fill [blue] ( 6/25 , 3/25 ) circle[radius=0.2pt];
\fill [blue] ( 6/25 , 4/25 ) circle[radius=0.2pt];
\fill [blue] ( 1/5 , 0 ) circle[radius=0.2pt];
\fill [blue] ( 1/5 , 1/30 ) circle[radius=0.2pt];
\fill [blue] ( 1/5 , 1/15 ) circle[radius=0.2pt];
\fill [blue] ( 1/5 , 1/10 ) circle[radius=0.2pt];
\fill [blue] ( 1/5 , 2/15 ) circle[radius=0.2pt];
\fill [blue] ( 1/6 , 0 ) circle[radius=0.2pt];
\fill [blue] ( 1/6 , 1/36 ) circle[radius=0.2pt];
\fill [blue] ( 1/6 , 1/18 ) circle[radius=0.2pt];
\fill [blue] ( 1/6 , 1/12 ) circle[radius=0.2pt];
\fill [blue] ( 1/6 , 1/9 ) circle[radius=0.2pt];
\fill [blue] ( 1/6 , 5/36 ) circle[radius=0.2pt];

\fill [blue,opacity = 0.2,line width=0.1mm](0,0)--(1,0)--(3/4,1/4)--( 3/8 , 1/4 )--( 1/4 , 7/36 )--( 3/16 , 5/32 )--( 3/20 , 13/100 )--( 1/8 , 1/9 )--( 3/28 , 19/196 )--( 3/32 , 11/128 )--( 1/12 , 25/324 )--( 3/40 , 7/100 )--( 3/44 , 31/484 )--
( 1/16 , 17/288 )--( 3/52 , 37/676 )--( 3/56 , 5/98 )--( 1/20 , 43/900 )--cycle;
\draw [blue,line width=0.5mm](0,0)--(1,0)--(3/4,1/4)--( 3/8 , 1/4 )--( 1/4 , 7/36 )--( 3/16 , 5/32 )--( 3/20 , 13/100 )--( 1/8 , 1/9 )--( 3/28 , 19/196 )--( 3/32 , 11/128 )--( 1/12 , 25/324 )--( 3/40 , 7/100 )--( 3/44 , 31/484 )--
( 1/16 , 17/288 )--( 3/52 , 37/676 )--( 3/56 , 5/98 )--( 1/20 , 43/900 )--cycle;

\end{tikzpicture}
\label{fig:summary}

\caption{A summary of our main results. The axes are $x=n/pq$ and $y=j/pq$. We show that if $(x,y)$ is outside  the shaded region bounded by $y = 1-x$, $y = x$, and $y = 1/3$, then $H_j [ \config{n}{p}{q} ] = 0$. We show conversely that for every rational point $(x,y)$ in the blue part of the shaded region, there exist $n,j,p,$ and $q$ such that $x=n / pq$, $y = j / pq$, and $H_j [ \config{n}{p}{q} ] \neq 0 $. Each of the blue dots represents a point $(x,y)$ where we computed that $H_j [ \config{n}{p}{q} ] \neq 0$, with $n \le 6$.}
\label{fig:main}
\end{figure}

%% file: npq222.txt
\begin{figure}

\begin{tikzpicture}[scale=0.35]
\draw[fill=gray] (0,8)--(5,5)--(8,0)--(3,3)--cycle;
\draw[fill=gray] (8,0)--(5,-5)--(0,-8)--(3,-3)--cycle;
\draw[fill=gray] (0,-8)--(-5,-5)--(-8,0)--(-3,-3)--cycle;
\draw[fill=gray] (-8,0)--(-5,5)--(0,8)--(-3,3)--cycle;

\draw (-0.5,8)--(-0.5,9)--(0.5,9)--(0.5,8)--cycle;
\draw (8,-0.5)--(9,-0.5)--(9,0.5)--(8,0.5)--cycle;
\draw (-0.5,-8)--(-0.5,-9)--(0.5,-9)--(0.5,-8)--cycle;
\draw (-8,-0.5)--(-9,-0.5)--(-9,0.5)--(-8,0.5)--cycle;
\draw (5,5)--(6,5)--(6,6)--(5,6)--cycle;
\draw (2,2)--(3,2)--(3,3)--(2,3)--cycle;
\draw (5,-5)--(6,-5)--(6,-6)--(5,-6)--cycle;
\draw (2,-2)--(3,-2)--(3,-3)--(2,-3)--cycle;
\draw (-5,-5)--(-6,-5)--(-6,-6)--(-5,-6)--cycle;
\draw (-2,-2)--(-3,-2)--(-3,-3)--(-2,-3)--cycle;
\draw (-5,5)--(-6,5)--(-6,6)--(-5,6)--cycle;
\draw (-2,2)--(-3,2)--(-3,3)--(-2,3)--cycle;

\draw[fill=darkgray](-0.5,8.5)--(-0.5,9)--(0,9)--(0,8.5)--cycle;
\draw[fill=lightgray] (0,8)--(0,8.5)--(0.5,8.5)--(0.5,8)--cycle;
\draw[fill=darkgray](8.5,0)--(9,0)--(9,0.5)--(8.5,0.5)--cycle;
\draw[fill=lightgray](8,-0.5)--(8.5,-0.5)--(8.5,0)--(8,0)--cycle;
\draw[fill=darkgray] (0,-8.5)--(0,-9)--(0.5,-9)--(0.5,-8.5)--cycle;
\draw[fill=lightgray] (-0.5,-8)--(-0.5,-8.5)--(0,-8.5)--(0,-8)--cycle;
\draw[fill=darkgray]  (-8.5,-0.5)--(-9,-0.5)--(-9,0)--(-8.5,0)--cycle;
\draw[fill=lightgray]  (-8,0)--(-8.5,0)--(-8.5,0.5)--(-8,0.5)--cycle;
\draw[fill=darkgray] (5,5.5)--(5.5,5.5)--(5.5,6)--(5,6)--cycle;
\draw[fill=lightgray] (5,5)--(5.5,5)--(5.5,5.5)--(5,5.5)--cycle;
\draw[fill=darkgray] (2.5,2.5)--(3,2.5)--(3,3)--(2.5,3)--cycle;
\draw[fill=lightgray] (2.5,2)--(3,2)--(3,2.5)--(2.5,2.5)--cycle;
\draw[fill=darkgray] (5.5,-5)--(6,-5)--(6,-5.5)--(5.5,-5.5)--cycle;
\draw[fill=lightgray] (5,-5)--(5.5,-5)--(5.5,-5.5)--(5,-5.5)--cycle;
\draw[fill=darkgray] (2.5,-2.5)--(3,-2.5)--(3,-3)--(2.5,-3)--cycle;
\draw[fill=lightgray] (2,-2.5)--(2.5,-2.5)--(2.5,-3)--(2,-3)--cycle;
\draw[fill=darkgray] (-5,-5.5)--(-5.5,-5.5)--(-5.5,-6)--(-5,-6)--cycle;
\draw[fill=lightgray] (-5,-5)--(-5.5,-5)--(-5.5,-5.5)--(-5,-5.5)--cycle;
\draw[fill=darkgray] (-2.5,-2.5)--(-3,-2.5)--(-3,-3)--(-2.5,-3)--cycle;
\draw [fill=lightgray] (-2.5,-2)--(-3,-2)--(-3,-2.5)--(-2.5,-2.5)--cycle;
\draw[fill=darkgray] (-5.5,5)--(-6,5)--(-6,5.5)--(-5.5,5.5)--cycle;
\draw [fill=lightgray] (-5,5)--(-5.5,5)--(-5.5,5.5)--(-5,5.5)--cycle;
\draw[fill=darkgray] (-2.5,2.5)--(-3,2.5)--(-3,3)--(-2.5,3)--cycle;
\draw [fill=lightgray] (-2,2.5)--(-2.5,2.5)--(-2.5,3)--(-2,3)--cycle;

\end{tikzpicture}
\caption{An illustration of the cell complex $\cell{2}{2}{2}$. The vertices of the complex are labeled by their corresponding configurations with integer coordinates. Note that in this simple case, the cell complex $\cell{2}{2}{2}$ equals the configuration space $\config{2}{2}{2}$, while in general the cell complex $\cell{n}{p}{q}$ is only a subspace of the configuration space $\config{n}{p}{q}$.}
\end{figure}

%% file: cycles.txt
\begin{figure}
\begin{tikzpicture}[scale=2]
\draw[line width=0.25mm] (0,0) circle (1.25);
\foreach \a in {0, ..., 7}
{
\draw[line width=0.25mm,fill=white](\a*45:1.25)++(-0.2,0.2)--++(0.4,0)--++(0,-0.4)--++(-0.4,0)--cycle;
}
\draw[line width=0.25mm,fill=lightgray] (-45:1.25)++(-0.2,0.2)--++(0.2,0)--++(0,-0.2)--++(-0.2,0)--cycle;
\draw[line width=0.25mm,fill=darkgray] (-45:1.25)--++(0.2,0)--++(0,-0.2)--++(-0.2,0)--cycle;
\draw[line width=0.25mm,fill=lightgray] (0:1.25)++(-0.2,0)--++(0.2,0)--++(0,-0.2)--++(-0.2,0)--cycle;
\draw[line width=0.25mm,fill=darkgray] (0:1.25)--++(0.2,0)--++(0,-0.2)--++(-0.2,0)--cycle;
\draw[line width=0.25mm,fill=lightgray] (45:1.25)++(-0.2,0)--++(0.2,0)--++(0,-0.2)--++(-0.2,0)--cycle;
\draw[line width=0.25mm,fill=darkgray] (45:1.25)++(0,0.2)--++(0.2,0)--++(0,-0.2)--++(-0.2,0)--cycle;
\draw[line width=0.25mm,fill=lightgray] (90:1.25)++(0,0)--++(0.2,0)--++(0,-0.2)--++(-0.2,0)--cycle;
\draw[line width=0.25mm,fill=darkgray] (90:1.25)++(0,0.2)--++(0.2,0)--++(0,-0.2)--++(-0.2,0)--cycle;
\draw[line width=0.25mm,fill=lightgray] (135:1.25)++(0,0)--++(0.2,0)--++(0,-0.2)--++(-0.2,0)--cycle;
\draw[line width=0.25mm,fill=darkgray] (135:1.25)++(-0.2,0.2)--++(0.2,0)--++(0,-0.2)--++(-0.2,0)--cycle;
\draw[line width=0.25mm,fill=lightgray] (180:1.25)++(0,0.2)--++(0.2,0)--++(0,-0.2)--++(-0.2,0)--cycle;
\draw[line width=0.25mm,fill=darkgray] (180:1.25)++(-0.2,0.2)--++(0.2,0)--++(0,-0.2)--++(-0.2,0)--cycle;
\draw[line width=0.25mm,fill=lightgray] (225:1.25)++(0,0.2)--++(0.2,0)--++(0,-0.2)--++(-0.2,0)--cycle;
\draw[line width=0.25mm,fill=darkgray] (225:1.25)++(-0.2,0)--++(0.2,0)--++(0,-0.2)--++(-0.2,0)--cycle;
\draw[line width=0.25mm,fill=lightgray] (270:1.25)++(-0.2,0.2)--++(0.2,0)--++(0,-0.2)--++(-0.2,0)--cycle;
\draw[line width=0.25mm,fill=darkgray] (270:1.25)++(-0.2,0)--++(0.2,0)--++(0,-0.2)--++(-0.2,0)--cycle;
\end{tikzpicture}
\caption{A nontrivial cycle in $H_1[ \config{2}{2}{2} ]$. This realizes the point $(x,y)=(1/2, 1/4)$.}
\label{fig:222}
\end{figure}

%% file: cycles2.txt
\begin{figure}

\begin{tikzpicture}[scale=2]
\draw[line width=0.25mm] (0,0) circle (1.25);
\foreach \a in {0, ..., 11}
{
\draw[line width=0.25mm,fill=white](\a*30:1.25)++(-0.2,0.2)--++(0.4,0)--++(0,-0.4)--++(-0.4,0)--cycle;
}
\draw[line width=0.25mm,fill=gray] (0:1.25)++(-0.2,0.2)--++(0.2,0)--++(0,-0.2)--++(-0.2,0)--cycle;
\draw[line width=0.25mm,fill=lightgray] (0:1.25)++(0,0.2)--++(0.2,0)--++(0,-0.2)--++(-0.2,0)--cycle;
\draw[line width=0.25mm,fill=darkgray] (0:1.25)--++(0.2,0)--++(0,-0.2)--++(-0.2,0)--cycle;
\draw[line width=0.25mm,fill=gray] (-30:1.25)++(-0.2,0.2)--++(0.2,0)--++(0,-0.2)--++(-0.2,0)--cycle;
\draw[line width=0.25mm,fill=lightgray] (-30:1.25)++(0,0.2)--++(0.2,0)--++(0,-0.2)--++(-0.2,0)--cycle;
\draw[line width=0.25mm,fill=darkgray] (-30:1.25)++(-0.2,0)--++(0.2,0)--++(0,-0.2)--++(-0.2,0)--cycle;
\draw[line width=0.25mm,fill=gray] (-60:1.25)++(-0.2,0.2)--++(0.2,0)--++(0,-0.2)--++(-0.2,0)--cycle;
\draw[line width=0.25mm,fill=lightgray] (-60:1.25)++(0,0)--++(0.2,0)--++(0,-0.2)--++(-0.2,0)--cycle;
\draw[line width=0.25mm,fill=darkgray] (-60:1.25)++(-0.2,0)--++(0.2,0)--++(0,-0.2)--++(-0.2,0)--cycle;
\draw[line width=0.25mm,fill=gray] (-90:1.25)++(0,0.2)--++(0.2,0)--++(0,-0.2)--++(-0.2,0)--cycle;
\draw[line width=0.25mm,fill=lightgray] (-90:1.25)++(0,0)--++(0.2,0)--++(0,-0.2)--++(-0.2,0)--cycle;
\draw[line width=0.25mm,fill=darkgray] (-90:1.25)++(-0.2,0)--++(0.2,0)--++(0,-0.2)--++(-0.2,0)--cycle;
\draw[line width=0.25mm,fill=gray] (-120:1.25)++(0,0.2)--++(0.2,0)--++(0,-0.2)--++(-0.2,0)--cycle;
\draw[line width=0.25mm,fill=lightgray] (-120:1.25)++(0,0)--++(0.2,0)--++(0,-0.2)--++(-0.2,0)--cycle;
\draw[line width=0.25mm,fill=darkgray] (-120:1.25)++(-0.2,0.2)--++(0.2,0)--++(0,-0.2)--++(-0.2,0)--cycle;
\draw[line width=0.25mm,fill=gray] (-150:1.25)++(0,0.2)--++(0.2,0)--++(0,-0.2)--++(-0.2,0)--cycle;
\draw[line width=0.25mm,fill=lightgray] (-150:1.25)++(-0.2,0)--++(0.2,0)--++(0,-0.2)--++(-0.2,0)--cycle;
\draw[line width=0.25mm,fill=darkgray] (-150:1.25)++(-0.2,0.2)--++(0.2,0)--++(0,-0.2)--++(-0.2,0)--cycle;
\draw[line width=0.25mm,fill=gray] (180:1.25)++(0,0)--++(0.2,0)--++(0,-0.2)--++(-0.2,0)--cycle;
\draw[line width=0.25mm,fill=lightgray] (180:1.25)++(-0.2,0)--++(0.2,0)--++(0,-0.2)--++(-0.2,0)--cycle;
\draw[line width=0.25mm,fill=darkgray] (180:1.25)++(-0.2,0.2)--++(0.2,0)--++(0,-0.2)--++(-0.2,0)--cycle;
\draw[line width=0.25mm,fill=gray] (150:1.25)++(0,0)--++(0.2,0)--++(0,-0.2)--++(-0.2,0)--cycle;
\draw[line width=0.25mm,fill=lightgray] (150:1.25)++(-0.2,0)--++(0.2,0)--++(0,-0.2)--++(-0.2,0)--cycle;
\draw[line width=0.25mm,fill=darkgray] (150:1.25)++(0,0.2)--++(0.2,0)--++(0,-0.2)--++(-0.2,0)--cycle;
\draw[line width=0.25mm,fill=gray] (120:1.25)++(0,0)--++(0.2,0)--++(0,-0.2)--++(-0.2,0)--cycle;
\draw[line width=0.25mm,fill=lightgray] (120:1.25)++(-0.2,0.2)--++(0.2,0)--++(0,-0.2)--++(-0.2,0)--cycle;
\draw[line width=0.25mm,fill=darkgray] (120:1.25)++(0,0.2)--++(0.2,0)--++(0,-0.2)--++(-0.2,0)--cycle;
\draw[line width=0.25mm,fill=gray] (90:1.25)++(-0.2,0)--++(0.2,0)--++(0,-0.2)--++(-0.2,0)--cycle;
\draw[line width=0.25mm,fill=lightgray] (90:1.25)++(-0.2,0.2)--++(0.2,0)--++(0,-0.2)--++(-0.2,0)--cycle;
\draw[line width=0.25mm,fill=darkgray] (90:1.25)++(0,0.2)--++(0.2,0)--++(0,-0.2)--++(-0.2,0)--cycle;
\draw[line width=0.25mm,fill=gray] (60:1.25)++(-0.2,0)--++(0.2,0)--++(0,-0.2)--++(-0.2,0)--cycle;
\draw[line width=0.25mm,fill=lightgray] (60:1.25)++(-0.2,0.2)--++(0.2,0)--++(0,-0.2)--++(-0.2,0)--cycle;
\draw[line width=0.25mm,fill=darkgray] (60:1.25)++(0,0)--++(0.2,0)--++(0,-0.2)--++(-0.2,0)--cycle;
\draw[line width=0.25mm,fill=gray] (30:1.25)++(-0.2,0)--++(0.2,0)--++(0,-0.2)--++(-0.2,0)--cycle;
\draw[line width=0.25mm,fill=lightgray] (30:1.25)++(0,0.2)--++(0.2,0)--++(0,-0.2)--++(-0.2,0)--cycle;
\draw[line width=0.25mm,fill=darkgray] (30:1.25)++(0,0)--++(0.2,0)--++(0,-0.2)--++(-0.2,0)--cycle;
\end{tikzpicture}
\caption{A nontrivial cycle in $H_1[\config{3}{2}{2}]$. This realizes the point $(x,y)=(3/4, 1/4)$.}
\label{fig:322}
\end{figure}

%% file: cycles3.txt
\begin{figure}

\begin{tikzpicture}[scale=2]
\draw[line width=0.25mm] (0,0) circle (1.25);
\foreach \a in {0, ..., 7}
{
\draw[blue,line width=0.25mm,fill=white](\a*45:1.25)++(-0.3,0.3)--++(0.6,0)--++(0,-0.6)--++(-0.6,0)--cycle;
}
\draw[blue,line width=0.25mm,fill=white] (-45:1.25)++(0.1,-0.1)--++(-0.4,0)--++(0,0.4)--++(0.4,0)--cycle;
\draw[line width=0.25mm,fill=black] (-45:1.25)++(0.1,-0.1)--++(0.2,0)--++(0,-0.2)--++(-0.2,0)--cycle;
\draw[blue,line width=0.25mm,fill=white] (0:1.25)++(0.1,-0.3)--++(-0.4,0)--++(0,0.4)--++(0.4,0)--cycle;
\draw[line width=0.25mm,fill=black] (0:1.25)++(0.1,-0.1)--++(0.2,0)--++(0,-0.2)--++(-0.2,0)--cycle;
\draw[blue,line width=0.25mm,fill=white] (45:1.25)++(0.1,-0.3)--++(-0.4,0)--++(0,0.4)--++(0.4,0)--cycle;
\draw[line width=0.25mm,fill=black] (45:1.25)++(0.1,0.3)--++(0.2,0)--++(0,-0.2)--++(-0.2,0)--cycle;
\draw[blue,line width=0.25mm,fill=white] (90:1.25)++(0.3,-0.3)--++(-0.4,0)--++(0,0.4)--++(0.4,0)--cycle;
\draw[line width=0.25mm,fill=black] (90:1.25)++(0.1,0.3)--++(0.2,0)--++(0,-0.2)--++(-0.2,0)--cycle;
\draw[blue,line width=0.25mm,fill=white] (135:1.25)++(0.3,-0.3)--++(-0.4,0)--++(0,0.4)--++(0.4,0)--cycle;
\draw[line width=0.25mm,fill=black] (135:1.25)++(-0.3,0.3)--++(0.2,0)--++(0,-0.2)--++(-0.2,0)--cycle;
\draw[blue,line width=0.25mm,fill=white] (180:1.25)++(0.3,-0.1)--++(-0.4,0)--++(0,0.4)--++(0.4,0)--cycle;
\draw[line width=0.25mm,fill=black] (180:1.25)++(-0.3,0.3)--++(0.2,0)--++(0,-0.2)--++(-0.2,0)--cycle;
\draw[blue,line width=0.25mm,fill=white] (225:1.25)++(0.3,-0.1)--++(-0.4,0)--++(0,0.4)--++(0.4,0)--cycle;
\draw[line width=0.25mm,fill=black] (225:1.25)++(-0.3,-0.1)--++(0.2,0)--++(0,-0.2)--++(-0.2,0)--cycle;
\draw[blue,line width=0.25mm,fill=white] (270:1.25)++(0.1,-0.1)--++(-0.4,0)--++(0,0.4)--++(0.4,0)--cycle;
\draw[line width=0.25mm,fill=black] (270:1.25)++(-0.3,-0.1)--++(0.2,0)--++(0,-0.2)--++(-0.2,0)--cycle;

\end{tikzpicture}
\caption{A $2 \times 2$ square and $1 \times 1$ orbiting each other in a $3 \times 3$ grid.}
\label{fig:2and1}
\end{figure}

%% file: torus-map.txt
\begin{figure}
\centering

\begin{tikzpicture}[line width=0.25mm,scale=0.30]

\foreach \a in {0, ..., 7}
{
\draw (-1,\a*4+1.5)--(32,\a*4+1.5);
\draw (\a*4+1.5,-1)--(\a*4+1.5,32);
}

\foreach \a in {0, ..., 7}
{
\foreach \b in {0, ..., 7}
{
\draw[fill=white] (\a*4,\b*4)--++(3,0)--++(0,3)--++(-3,0)--cycle;
}
}

\foreach \a in {0, ..., 7}
{
\draw[fill=black] (2,\a*4)--++(1,0)--++(0,1)--++(-1,0)--cycle;
\draw[fill=black] (6,\a*4+2)--++(1,0)--++(0,1)--++(-1,0)--cycle;
\draw[fill=black] (10,\a*4+2)--++(1,0)--++(0,1)--++(-1,0)--cycle;
\draw[fill=black] (12,\a*4+2)--++(1,0)--++(0,1)--++(-1,0)--cycle;
\draw[fill=black] (16,\a*4+2)--++(1,0)--++(0,1)--++(-1,0)--cycle;
\draw[fill=black] (20,\a*4)--++(1,0)--++(0,1)--++(-1,0)--cycle;
\draw[fill=black] (24,\a*4)--++(1,0)--++(0,1)--++(-1,0)--cycle;
\draw[fill=black] (30,\a*4)--++(1,0)--++(0,1)--++(-1,0)--cycle;
\ifthenelse{\a=2 \OR \a=3 \OR \a=4 \OR \a=5}{\def\x{1}}{\def\x{0}}
\ifthenelse{\a=6 \OR \a=7 \OR \a=4 \OR \a=5}{\def\y{1}}{\def\y{0}}
\draw[fill=lightgray] (\a*4+\x,\y)--++(1,0)--++(0,1)--++(-1,0)--cycle;
\draw[fill=gray] (\a*4+\x+1,\y)--++(1,0)--++(0,1)--++(-1,0)--cycle;
\draw[fill=lightgray] (\a*4+\x,4+\y)--++(1,0)--++(0,1)--++(-1,0)--cycle;
\draw[fill=gray] (\a*4+\x+1,5+\y)--++(1,0)--++(0,1)--++(-1,0)--cycle;
\draw[fill=lightgray] (\a*4+\x+1,8+\y)--++(1,0)--++(0,1)--++(-1,0)--cycle;
\draw[fill=gray] (\a*4+\x+1,9+\y)--++(1,0)--++(0,1)--++(-1,0)--cycle;
\draw[fill=lightgray] (\a*4+\x+1,12+\y)--++(1,0)--++(0,1)--++(-1,0)--cycle;
\draw[fill=gray] (\a*4+\x,13+\y)--++(1,0)--++(0,1)--++(-1,0)--cycle;
\draw[fill=lightgray] (\a*4+\x+1,17+\y)--++(1,0)--++(0,1)--++(-1,0)--cycle;
\draw[fill=gray] (\a*4+\x,17+\y)--++(1,0)--++(0,1)--++(-1,0)--cycle;
\draw[fill=lightgray] (\a*4+\x+1,21+\y)--++(1,0)--++(0,1)--++(-1,0)--cycle;
\draw[fill=gray] (\a*4+\x,20+\y)--++(1,0)--++(0,1)--++(-1,0)--cycle;
\draw[fill=lightgray] (\a*4+\x,25+\y)--++(1,0)--++(0,1)--++(-1,0)--cycle;
\draw[fill=gray] (\a*4+\x,24+\y)--++(1,0)--++(0,1)--++(-1,0)--cycle;
\draw[fill=lightgray] (\a*4+\x,29+\y)--++(1,0)--++(0,1)--++(-1,0)--cycle;
\draw[fill=gray] (\a*4+\x+1,28+\y)--++(1,0)--++(0,1)--++(-1,0)--cycle;
}

\end{tikzpicture}
\caption{Another view of the map $i: T^2 \to \config{3}{3}{3}$ visualized in Figure \ref{fig:2333}. The image of the fundamental class of the torus is a nontrivial cycle in $H_2[ \config{3}{3}{3}]$.}
\label{fig:2torus}
\end{figure}

%% file: computations.tex

\section{Betti number computations for small $n,p,q$}\label{sec:computation}

We compute the Betti numbers $\beta_j[C(n;p,q)]$ for $n\leq 6$ and $p\leq q\leq n$.  These are provided in Table~\ref{table:bettis}.  Another view of the Betti numbers for $n=6$ and $j=2$ with  $p$ and $q$ varying is illustrated in Figure~\ref{fig:n6Betti}.  Finally, in  Table~\ref{table:fvectors} we record information about the size of the complex $\cell{n}{p}{q}$ in the form of its \emph{$f$-vector} $(f_0, f_1, f_2, \dots)$, where $f_i$ is the number of $i$-dimensional cells in $\cell{n}{p}{q}$. 
All of our computations are using coefficients in the prime field $\Z/2\Z$.

For our computations we employ three different software packages, and we dedicate a small section to each one. The first is a Python/Sage Jupyter notebook which uses the discrete Morse vector field of Section~\ref{sec:dMt}.  The second is a branch of the \textsc{PyCHomP} package, available at~\cite{pychomp} specifically for computing the Betti numbers for these configuration spaces. The third is the \textsc{Dipha} package with a custom script to build the configuration cell complex.
Finally, note that in the case when $n=q$ the configuration space $\config{n}{p}{q}$ is homotopy equivalent to the configuration space of disks in a strip addressed in~\cite{AKM19}; in this case, one can use the Salvetti complex to compute the Betti numbers as done in~\cite{AKM19}.

\begin{table}[h!]
\caption{The Betti numbers of $C(n;p,q)$ for $2 \le n \le 6$. The homological liquid regime is indicated in bold.}
\begin{tabular}{|c|c|c||c|c|c|c|c|c|} \hline
$n$ & $p$ & $q$ & $\beta_0$ & $\beta_1$ & $\beta_2$ & $\beta_3$ & $\beta_4$ & $\beta_5$ \\ \hline
2 & 2 & 2 & 1 & 1 & 0 & 0 & 0 & 0 \\ \hline
3 & 2 & 2 & \textbf{2} & \textbf{2} & 0 & 0 & 0 & 0 \\
3 & 2 & 3 & 1 & \textbf{7} & 0 & 0 & 0 & 0 \\
3 & 3 & 3 & 1 & 3 & 2 & 0 & 0 & 0 \\ \hline
4 & 2 & 2 & \textbf{24} & 0 & 0 & 0 & 0 & 0 \\
4 & 2 & 3 & 1 & \textbf{49} & 0 & 0 & 0 & 0 \\
4 & 2 & 4 & 1 & \textbf{31} & \textbf{6} & 0 & 0 & 0 \\
4 & 3 & 3 & 1 & \textbf{12} & \textbf{11} & 0 & 0 & 0 \\
4 & 3 & 4 & 1 & 6 & \textbf{29} & 0 & 0 & 0 \\
4 & 4 & 4 & 1 & 6 & 11 & 6 & 0 & 0 \\ \hline
5 & 2 & 3 & \textbf{2} & \textbf{122} & 0 & 0 & 0 & 0 \\
5 & 2 & 4 & 1 & \textbf{161} & \textbf{40} & 0 & 0 & 0 \\
5 & 2 & 5 & 1 & \textbf{111} & \textbf{110} & 0 & 0 & 0 \\
5 & 3 & 3 & 1 & \textbf{68} & \textbf{67} & 0 & 0 & 0 \\
5 & 3 & 4 & 1 & 10 & \textbf{249} & 0 & 0 & 0 \\
5 & 3 & 5 & 1 & 10 & \textbf{169} & \textbf{40} & 0 & 0 \\
5 & 4 & 4 & 1 & 10 & \textbf{71} & \textbf{62} & 0 & 0 \\
5 & 4 & 5 & 1 & 10 & 35 & \textbf{146} & 0 & 0 \\
5 & 5 & 5 & 1 & 10 & 35 & 50 & 24 & 0 \\ \hline
6 & 2 & 3 & \textbf{720} & 0 & 0 & 0 & 0 & 0 \\
6 & 2 & 4 & 1 & \textbf{2241} & \textbf{80} & 0 & 0 & 0 \\
6 & 2 & 5 & 1 & \textbf{351} & \textbf{1790} & 0 & 0 & 0 \\
6 & 2 & 6 & 1 & \textbf{351} & \textbf{1160} & \textbf{90} & 0 & 0 \\
6 & 3 & 3 & 1 & \textbf{458} & \textbf{457} & 0 & 0 & 0 \\
6 & 3 & 4 & 1 & 15 & \textbf{2174} & 0 & 0 & 0 \\
6 & 3 & 5 & 1 & 15 & \textbf{714} & \textbf{1429} & 0 & 0 \\
6 & 3 & 6 & 1 & 15 & \textbf{714} & \textbf{780} & \textbf{80} & 0 \\
6 & 4 & 4 & 1 & 15 & \textbf{441} & \textbf{457} & \textbf{30} & 0 \\
6 & 4 & 5 & 1 & 15 & 85 & \textbf{1541} & \textbf{30} & 0 \\
6 & 4 & 6 & 1 & 15 & 85 & \textbf{1066} & \textbf{275} & 0 \\
6 & 5 & 5 & 1 & 15 & 85 & \textbf{465} & \textbf{394} & 0 \\
6 & 5 & 6 & 1 & 15 & 85 & 225 & \textbf{875} & 0 \\
6 & 6 & 6 & 1 & 15 & 85 & 225 & 274 & 120 \\ \hline
\end{tabular}
\label{table:bettis}
\end{table}

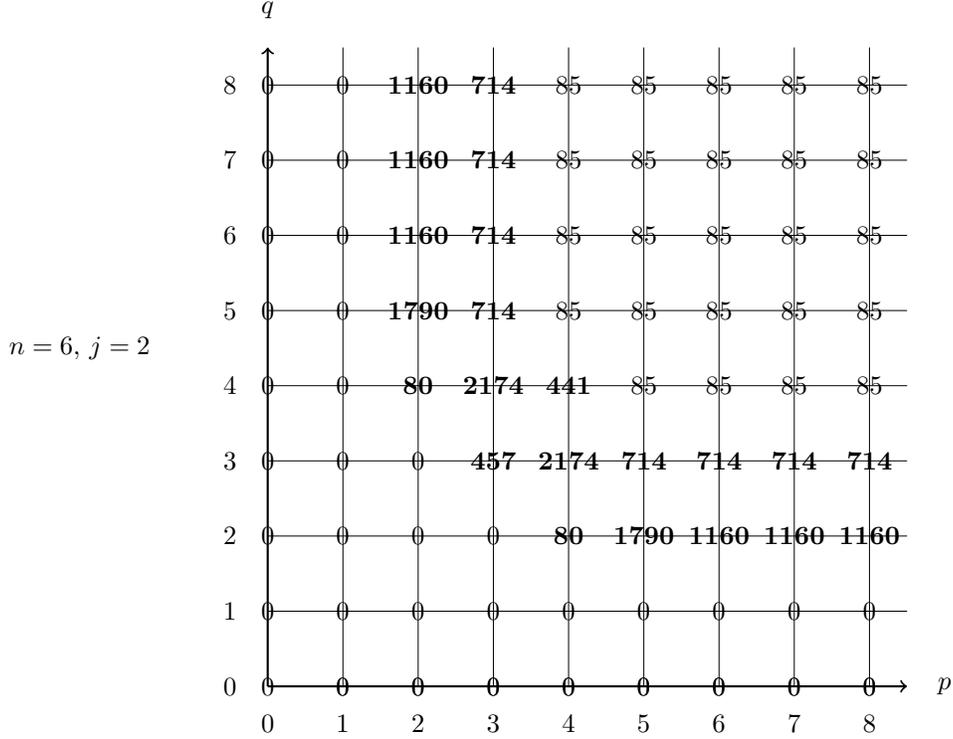
\begin{figure}

\centering

\begin{tikzpicture}[thick]

\draw [->] (-0.5,-0.5) -- (8,-0.5);
\node at (8.5, -0.5) {$p$};


\draw[line width = 0.1mm] (-0.5,0.5) -- (8,0.5);
\draw[line width = 0.1mm] (-0.5,1.5) -- (8,1.5);
\draw[line width = 0.1mm] (-0.5,2.5) -- (8,2.5);
\draw[line width = 0.1mm] (-0.5,3.5) -- (8,3.5);
\draw[line width = 0.1mm] (-0.5,4.5) -- (8,4.5);
\draw[line width = 0.1mm] (-0.5,5.5) -- (8,5.5);
\draw[line width = 0.1mm] (-0.5,6.5) -- (8,6.5);
\draw[line width = 0.1mm] (-0.5,7.5) -- (8,7.5);

\draw[line width = 0.1mm] (0.5,-0.5) -- (0.5,8);
\draw[line width = 0.1mm] (1.5,-0.5) -- (1.5,8);
\draw[line width = 0.1mm] (2.5,-0.5) -- (2.5,8);
\draw[line width = 0.1mm] (3.5,-0.5) -- (3.5,8);
\draw[line width = 0.1mm] (4.5,-0.5) -- (4.5,8);
\draw[line width = 0.1mm] (5.5,-0.5) -- (5.5,8);
\draw[line width = 0.1mm] (6.5,-0.5) -- (6.5,8);
\draw[line width = 0.1mm] (7.5,-0.5) -- (7.5,8);

\draw[->]  (-0.5,-0.5) -- (-0.5,8);
\node at (-0.5,8.5) {$q$};

\node at (-3,4) {\bf  $n=6, \, j=2$};

\node at (-1/2, -1) {$0$};
\node at (1/2, -1) {$1$};
\node at (3/2, -1) {$2$};
\node at (5/2, -1) {$3$};
\node at (7/2, -1) {$4$};
\node at (9/2, -1) {$5$};
\node at (11/2, -1) {$6$};
\node at (13/2, -1) {$7$};
\node at (15/2, -1) {$8$};

\node at (-1,-1/2) {$0$};
\node at (-1,1/2) {$1$};
\node at (-1,3/2) {$2$};
\node at (-1,5/2) {$3$};
\node at (-1,7/2) {$4$};
\node at (-1,9/2) {$5$};
\node at (-1,11/2) {$6$};
\node at (-1,13/2) {$7$};
\node at (-1,15/2) {$8$};

\node at (-1/2,-1/2) {$0$};
\node at (1/2,-1/2) {$0$};
\node at (3/2,-1/2) {$0$};
\node at (5/2,-1/2) {$0$};
\node at (7/2,-1/2) {$0$};
\node at (9/2,-1/2) {$0$};
\node at (11/2,-1/2) {$0$};
\node at (13/2,-1/2) {$0$};
\node at (15/2,-1/2) {$0$};

\node at (-1/2,1/2) {$0$};
\node at (-1/2,3/2) {$0$};
\node at (-1/2,5/2) {$0$};
\node at (-1/2,7/2) {$0$};
\node at (-1/2,9/2) {$0$};
\node at (-1/2,11/2) {$0$};
\node at (-1/2,13/2) {$0$};
\node at (-1/2,15/2) {$0$};

\node at (1/2,1/2) {$0$};
\node at (3/2,1/2) {$0$};
\node at (5/2,1/2) {$0$};
\node at (7/2,1/2) {$0$};
\node at (9/2,1/2) {$0$};
\node at (11/2,1/2) {$0$};
\node at (13/2,1/2) {$0$};
\node at (15/2,1/2) {$0$};

\node at (1/2,-1/2) {$0$};
\node at (3/2,-1/2) {$0$};
\node at (5/2,-1/2) {$0$};
\node at (7/2,-1/2) {$0$};
\node at (9/2,-1/2) {$0$};
\node at (11/2,-1/2) {$0$};
\node at (13/2,-1/2) {$0$};
\node at (15/2,-1/2) {$0$};

\node at (1/2,3/2) {$0$};
\node at (3/2,3/2) {$0$};
\node at (5/2,3/2) {$0$};
\node at (7/2,3/2) {$\bf 80$};
\node at (9/2,3/2) {$\bf 1790$};
\node at (11/2,3/2) {$\bf 1160$};
\node at (13/2,3/2) {$\bf 1160$};
\node at (15/2,3/2) {$\bf 1160$};

\node at (1/2,5/2) {$0$};
\node at (3/2,5/2) {$0$};
\node at (5/2,5/2) {$\bf 457$};
\node at (7/2,5/2) {$\bf 2174$};
\node at (9/2,5/2) {$\bf 714$};
\node at (11/2,5/2) {$\bf 714$};
\node at (13/2,5/2) {$\bf 714$};
\node at (15/2,5/2) {$\bf 714$};

\node at (1/2,7/2) {$0$};
\node at (3/2,7/2) {$\bf 80$};
\node at (5/2,7/2) {$\bf 2174$};
\node at (7/2,7/2) {$\bf 441$};
\node at (9/2,7/2) {$85$};
\node at (11/2,7/2) {$85$};
\node at (13/2,7/2) {$85$};
\node at (15/2,7/2) {$85$};

\node at (1/2,9/2) {$0$};
\node at (3/2,9/2) {$\bf 1790$};
\node at (5/2,9/2) {$\bf 714$};
\node at (7/2,9/2) {$85$};
\node at (9/2,9/2) {$85$};
\node at (11/2,9/2) {$85$};
\node at (13/2,9/2) {$85$};
\node at (15/2,9/2) {$85$};

\node at (1/2,11/2) {$0$};
\node at (3/2,11/2) {$\bf 1160$};
\node at (5/2,11/2) {$\bf 714$};
\node at (7/2,11/2) {$85$};
\node at (9/2,11/2) {$85$};
\node at (11/2,11/2) {$85$};
\node at (13/2,11/2) {$85$};
\node at (15/2,11/2) {$85$};

\node at (1/2,13/2) {$0$};
\node at (3/2,13/2) {$\bf 1160$};
\node at (5/2,13/2) {$\bf 714$};
\node at (7/2,13/2) {$85$};
\node at (9/2,13/2) {$85$};
\node at (11/2,13/2) {$85$};
\node at (13/2,13/2) {$85$};
\node at (15/2,13/2) {$85$};

\node at (1/2,15/2) {$0$};
\node at (3/2,15/2) {$\bf 1160$};
\node at (5/2,15/2) {$\bf 714$};
\node at (7/2,15/2) {$85$};
\node at (9/2,15/2) {$85$};
\node at (11/2,15/2) {$85$};
\node at (13/2,15/2) {$85$};
\node at (15/2,15/2) {$85$};

\end{tikzpicture}
\caption{Another view of the Betti numbers. Let $n=6$ and $j=2$, and let $p$ and $q$ be the horizontal and vertical axes. Then the solid regime is in the lower left, the gas regime is in the upper right, and the liquid regime (in bold) is in between. If $p \ge n$, then the inclusion map $\config{n}{p}{q} \hookrightarrow \config{n}{p+1}{q}$ induces an isomorphism on homology. Similarly, if $q \ge n$ then the inclusion map $\config{n}{p}{q} \hookrightarrow \config{n}{p}{q+1}$ induces an isomorphism on homology.}\label{fig:n6Betti}
\end{figure}

\begin{table}[h!]
\centering\tiny
\begin{tabular}{|c|c|c||c|c|c|c|c|c|c|c|c|} \hline
$n$ & $p$ & $q$ & $f_0$ & $f_1$ & $f_2$ & $f_3$ & $f_4$ & $f_5$ & $f_6$ & $f_7$ & $f_8$ \\ \hline
2 & 2 & 2 & 12 & 16 & 4 & & & & & & \\ \hline
3 & 2 & 2 & 24 & 24 &  & & & & & & \\ 
3 & 2 & 3 & 120 & 252 & 144  & 18 & & & & & \\ 
3 & 3 & 3 & 504 & 1512 & 1560 & 624 & 72 & & & & \\ \hline
4 & 2 & 3 & 360 & 672 & 264  & & & & & & \\ 
4 & 2 & 4 & 1680 & 4800 & 4464 & 1488 & 120 & & & & \\ 
4 & 3 & 3 & 3024 & 10080 & 11520 & 5184 & 720 & & & &  \\ 
4 & 3 & 4 & 11880 & 48960 & 76608 & 56448 & 19536 & 2688 & 96 & &\\ 
4 & 4 & 4 & 43680 & 209664 & 402336 & 393120 & 206232 & 56640 & 7728 & 576 & 24 \\ \hline 
5 & 2 & 3 & 720 & 840 &  &  &  &  &  &  &  \\  
5 & 2 & 4 & 6720 & 18000 & 14280 & 3120 &  &  &  &  &  \\  
5 & 2 & 5 & 30240 & 109200 & 141600 & 79200 & 17520  &  960 &  &  &  \\ 
5 & 3 & 3 & 15120 & 50400 & 55200 & 22080 & 2160  &   &  &  &  \\ 
5 & 3 & 4 & 95040 & 428400 & 735840 & 600600 & 234720  &  38040 & 1680 &  &  \\ 
5 & 3 & 5 & 360360 & 1887600 & 3979800 & 4322880 & 2561160  &  800400 & 114960 & 5280 &  \\ 
5 & 4 & 4 & 524160 & 2882880 & 6448200 & 7538400 & 4928640  &  1793280 & 345240 & 33120 & 1440 \\ \hline
6 & 3 & 3 & 60480 & 181440 & 161280 & 40320 &&&&& \\ \hline
\end{tabular}
\bigskip
\caption{The $f$-vectors for $\cell{n}{p}{q}$ for small $n, p, q$.}
\label{table:fvectors}
\end{table}


\subsection{Discrete Morse Theory Sage Notebook}

Using the discrete gradient vector field from Section~\ref{sec:dMt}, we compute the collapsed Morse chain complex for $\cell{n}{p}{q}$ as follows.  The idea is first to find the critical cells and then to compute their boundaries in the Morse complex.  However, it turns out that most of this process depends very little on $p$ and $q$.  Thus, in order to compute for various $p$ and $q$ without duplicating effort, we first compute the Morse complex for $\cell{n}{n}{n}$.  The Morse complex of each $\cell{n}{p}{q}$ for $1 \leq p, q \leq n$ turns out to be a subcomplex of the Morse complex for $\cell{n}{n}{n}$, obtained by selecting only the critical cells for which the apex is in the $p$ by $q$ rectangle.  This is because of the properties of our discrete gradient vector field.  Namely, we know that if a cell's apex fits into a $p$ by $q$ rectangle, so does every boundary cell of that cell (the apex takes upper right corners, and the $p$ by $q$ rectangle grows from the lower left); together with the fact that every two paired cells have the same apex, this implies that the $\cell{n}{p}{q}$ Morse complex is a subcomplex of the $\cell{n}{p'}{q'}$ Morse complex whenever $p \leq p'$ and $q \leq q'$.  The construction of the discrete gradient vector field guarantees that no apex that skips a row or column can be the apex of a critical cell---this is because every apex graph with an isolated vertex has an even number of independent sets---so the $\cell{n}{n}{n}$ Morse complex is sufficiently large to contain the Morse complexes for all $\cell{n}{p}{q}$.

Thus, the code computes as follows.  First, we list all ways of placing $n$ squares in an $n$ by $n$ grid.  Then, we check which of these arrangements are the apex of a critical cell.  For each critical cell, we compute its boundary in the Morse complex by applying discrete gradient flow to its original boundary in $\cell{n}{n}{n}$.  Doing this for every critical cell gives all boundary coefficients for the Morse complex of $\cell{n}{n}{n}$, computed as integers with signs.  Then we restrict to smaller $p$ by $q$ rectangles, producing subcomplexes of the Morse complex.  For each $p$ and $q$, we compute the Betti numbers from the ranks of the boundary matrices and the dimensions of the chain groups; because the matrices have integer entries, to specify the coefficient field for homology, we only need to specify the field for the rank computation, which can be done over $\mathbb{Q}$ or modulo any choice of prime.
The Sage notebook is available online.\footnote{The Sage notebook containing the described code can be found at \url{https://gist.github.com/ubauer/87e7ee1462966127e9837c4747829a4a}.}

We found that the code runs quickly for $n \leq 5$ and agrees with our other computation methods; for $n \leq 6$ it becomes slow and would require more speed optimization.

\subsection{PyCHomP}

We briefly review the computations involved in \textsc{PyCHomP}, which may be used to compute the homology of $\cell{n}{p}{q}$ with $\Z/2\Z$ coefficients.

Let $(P,\leq)$ be the total order with $P=\{0,1\}$ and $0\leq 1$.  As $\cell{n}{p}{q}$ is a subcomplex of $\ambient{n}{p}{q}$, there is an order-preserving map $\nu$ from the face poset $(\ambient{n}{p}{q},\leq)$ to $(P,\leq)$ given by
\[
\nu(\sigma) = 
\begin{cases}
    0 & \text{ if $\sigma \in \cell{n}{p}{q}$ } \\
    1 & \text{ if $\sigma \not\in \cell{n}{p}{q}$.}
\end{cases}
\]

In order to construct the map $\nu$, we use Lemma~\ref{lem-partial-cell} to determine whether a cell belongs to the cubical complex $\cell{n}{p}{q}$. The complex $\ambient{n}{p}{q}$ together with the map $\nu$ defines a $P$-graded cell complex (see~\cite{HMS19}).   \textsc{PyCHomP} uses iterated algebraic-discrete Morse theory to reduce $\ambient{n}{p}{q}$ to a (chain-homotopy equivalent) $P$-graded cell complex $(A(n;p,q),\mu)$ characterized by the condition that
$\partial^A|_{\mu^{-1}(p)}=0$ for $p\in P$.  This condition implies that the $j$-dimensional Betti number of $\cell{n}{p}{q}$ is precisely the number of $j$-dimensional cells in $\mu^{-1}(0)$; see~\cite[Example 4.30]{HMS19} for more detail.

Theorems~\ref{thm:hv:pq-n}--\ref{thm:pq3} suggest that speed ups can be obtained for any code which computes homology starting from the complex $\cell{n}{p}{q}$ by not considering cells above a certain dimension.  The branch of \textsc{PyCHomP} available at~\cite{pychomp} incorporates these speed-ups; \textsc{PyCHomP} is able to compute the Betti numbers for all the examples in Table~\ref{table:bettis}.  A Jupyter notebook which sets up the computation of Betti numbers for $\cell{n}{p}{q}$ is available online.\footnote{The Jupyter notebook can be found at \url{https://github.com/kellyspendlove/pyCHomP/blob/config/doc/config/ConfigSpacePaper.ipynb}.}

\subsection{DIPHA}

Finally, we describe the homology computation of $\cell{n}{p}{q}$ using \textsc{Dipha}, a software package for computing persistent homology in a distributed setting \cite{Bauer2014Distributed}.
\textsc{Dipha} supports the computation of persistent homology for \emph{lower star filtrations} of cubical grids such as $\ambient{n}{p}{q}$. The data determining a lower star filtration is a real-valued function $f$ on the vertices of $\ambient{n}{p}{q}$, i.e., the integer points in $([1, p] \times [1, q])^n$.
The filtration then consists of the full subcomplexes of the ambient cube complex $\ambient{n}{p}{q}$ induced by sublevel sets $f^{-1}(-\infty,t]$ of the function~$f$.

Our computations make use of the fact that $\cell{n}{p}{q}$ is a full subcomplex of the ambient cube complex $\ambient{n}{p}{q}$ (see Corollary~\ref{cor-full-subcomplex}).
In other words, the complex $\cell{n}{p}{q}$ is determined by the set of all configurations in $\config{n}{p}{q}$ with integer coordinates.
Thus, it suffices to enumerate all permutations of all $n$-element subsets of the $p \times q$ possible integer coordinates for the cubes.
The input to \textsc{Dipha} consists of the characteristic function of this set as a subset of all vertices of $\ambient{n}{p}{q}$.
A Mathematica file for generating the input to \textsc{Dipha} is available online.\footnote{The Mathematica notebook for generating the \textsc{Dipha} input can be found at \url{https://gist.github.com/ubauer/01934ad494eeb6e9ef66ca14e0301fe9}.}